\documentclass[11pt,reqno]{amsart}

\usepackage{amsmath,amsthm,amssymb}
\usepackage{soul,url,tocvsec2}
\usepackage[foot]{amsaddr}
\usepackage[tmargin=1.4in,bmargin=1.4in,rmargin=1.4in,lmargin=1.4in]{geometry}
\usepackage[breaklinks=true]{hyperref}

\theoremstyle{plain}

\newtheorem{theorem}{Theorem}[section]
\newtheorem{corollary}[theorem]{Corollary}
\newtheorem{proposition}[theorem]{Proposition}
\newtheorem{lemma}[theorem]{Lemma}

\theoremstyle{definition}

\newtheorem{definition}[theorem]{Definition}

\newtheorem{remark}[theorem]{Remark}

\theoremstyle{remark}

\numberwithin{equation}{section}

\makeatletter
\let\c@theorem\c@table
\makeatother

\newcommand{\ba}{\mathbf{a}}
\newcommand{\bc}{\mathbf{c}}
\newcommand{\bd}{\mathbf{d}}
\newcommand{\bn}{\mathbf{n}}
\newcommand{\bp}{\mathcal{P}}

\newcommand{\bu}{\mathbf{u}}
\newcommand{\bv}{\mathbf{v}}
\newcommand{\bw}{\mathbf{w}}
\newcommand{\cals}{\mathcal{S}}
\newcommand{\diag}{\mathop{\mathrm{diag}}\nolimits}
\newcommand{\Gr}{\mathop{\mathrm{Gr}}\nolimits}
\newcommand{\Id}{\mathrm{Id}}
\newcommand{\rd}{\mathrm{d}}
\newcommand{\vi}{\varepsilon}
\newcommand{\one}[1]{\mathbf{1}_{#1 \times #1}}
\newcommand{\sgn}{\mathop{\mathrm{sgn}}\nolimits}
\newcommand{\std}{\,\rd}
\newcommand{\commen}[1]{}

\newcommand{\R}{\mathbb{R}}
\newcommand{\Q}{\mathbb{Q}}
\newcommand{\C}{\mathbb{C}}
\newcommand{\F}{\mathbb{F}}

\renewcommand{\Re}{\mathop{\mathrm{Re}}}

\begin{document}
\title{Matrix positivity preservers in fixed dimension. I}

\author{Alexander Belton$^*$}
\address{$^*$Lancaster University, Lancaster, UK;
email: {\tt a.belton@lancaster.ac.uk}}

\author{Dominique Guillot$^\dag$}
\address{$^\dag$University of Delaware, Newark, DE, USA;
email: {\tt dguillot@udel.edu}}

\author{Apoorva Khare$^\ddag$}
\address{$^\ddag$Stanford University, Stanford, CA, USA;
email: {\tt khare@stanford.edu}}

\author{Mihai Putinar$^\mathsection$}
\address{$^\mathsection$University of California at Santa Barbara, CA, USA and
Newcastle University, Newcastle upon Tyne, UK;
email: {\tt mputinar@math.ucsb.edu, mihai.putinar@ncl.ac.uk}}

\date{\today}



\begin{abstract}
A classical theorem proved in 1942 by I.J. Schoenberg describes all
real-valued functions that preserve positivity when applied entrywise
to positive semidefinite matrices of arbitrary size; such functions
are necessarily analytic with non-negative Taylor coefficients.
Despite the great deal of interest generated by this theorem, a
characterization of functions preserving positivity for matrices of
fixed dimension is not known.

In this paper, we provide a complete description of
polynomials of degree~$N$ that preserve positivity when applied
entrywise to matrices of dimension~$N$. This is the key step for us
then to obtain negative lower bounds on the coefficients
of analytic functions so that these functions preserve positivity in a
prescribed dimension. The proof of the main technical inequality is
representation theoretic, and employs the theory of Schur
polynomials. Interpreted in the context of linear pencils of matrices,
our main results provide a closed-form expression for the lowest
critical value, revealing at the same time an unexpected spectral
discontinuity phenomenon.

Tight linear matrix inequalities for Hadamard powers of matrices and a
sharp asymptotic bound for the matrix-cube problem involving Hadamard
powers are obtained as applications. Positivity preservers are also
naturally interpreted as solutions of a variational inequality
involving generalized Rayleigh quotients. This optimization approach
leads to a novel description of the simultaneous kernels of Hadamard
powers, and a family of stratifications of the cone of positive
semidefinite matrices.
\end{abstract}
\maketitle

\tableofcontents

\section{Introduction and main results}\label{Sintro}

Transformations, linear or not, which preserve matrix
structures with positivity constraints have been recently studied in at
least three distinct frameworks:
statistical mechanics and the geometry of polynomials
\cite{Borcea-Branden-1, Borcea-Branden-2,Branden};
global optimization algorithms based on the cone of hyperbolic or
positive definite polynomials \cite{SIAM-optimization, HMPV, Renegar};
the statistics of big data, having the correlation matrix of a large
number of random variables as the central object \cite{bickel_levina,
hero_rajaratnam, Li_Horvath, Rothman_et_al_JASA, Zhang_Horvath}.
The present article belongs in the latter two categories, although the
main result may be of independent algebraic interest.

To describe the contents of this paper, we adopt some
terminology. For a set $K \subset \C$ and an integer $N \geq 1$,
denote by $\bp_N( K )$ the cone of positive semidefinite $N \times N$
matrices with entries in $K$. A function $f: K \to \C$ naturally acts
entrywise on $\bp_N( K )$, so that $f[ A ] := ( f( a_{i j}) )$ for any
$A = ( a_{i j} ) \in \bp_N( K )$. Akin to the theory of positive
definite functions, it is natural to seek characterizations of those
functions $f$ such that $f[ A ]$ is positive semidefinite for all
$A \in \bp_N( K )$.
A well-known theorem of Schoenberg
\cite{Schoenberg42} states that $f[ A ]$ is positive semidefinite for
all $A \in \bp_N( [ {-1}, 1 ] )$ of all dimensions $N \geq 1$ if and
only if $f$ is absolutely monotonic on $[0,1]$ (i.e., analytic with
non-negative Taylor coefficients).
To put Schoenberg's 1942 article in historical perspective, we have to
recall that the theory of absolute monotone functions
was already established by S.~Bernstein \cite{Bernstein}.
Also, it is worth mentioning that Schoenberg was working around that time
on the related and more general question of isometrically embedding
positive definite metrics into Hilbert space; see, for instance,
\cite{vonNeumann-Schoenberg}.
The parallel theory of matrix monotone functions, with $f( A )$ defined
by standard functional calculus, owes its main
result to Loewner \cite{Loewner34} (see also \cite{Donoghue_74}).

Since its publication, Schoenberg's theorem has attracted a great
deal of attention.
The result has been considered in several different contexts in
\cite{BCR-semigroups, Bochner-pd, Bochner-zonal, Christensen_et_al78,
Herz63, Rudin59}.
See also \cite{Berg-Porcu, gneiting2013, GMP, Hiai2009, menegatto_2001}
for recent work.

Schoenberg's observation
has a natural application to high-dimensional probability and statistics.
Recall that a correlation matrix is the Gram matrix of vectors on a
sphere $S^{d - 1}$. In concrete situations, functions are often applied
entrywise to correlation matrices, in order to improve their properties,
such as better conditioning, or to induce a Markov random field
structure; among the recent investigations centered on this technique we
note
\cite{bickel_levina, Guillot_Rajaratnam2012, Guillot_Rajaratnam2012b,
hero_rajaratnam, Hero_Rajaratnam2012, Li_Horvath, Rothman_et_al_JASA,
Zhang_Horvath}.
Whether or not the resulting matrices are positive semidefinite is
critical for the validity of these procedures. According to
Schoenberg's theorem, functions preserving positivity when applied
entrywise to correlation matrices of all dimensions and bounded rank
are non-negative combinations of Gegenbauer polynomials. However,
allowing for arbitrary dimension is unnecessarily restrictive, as the
state space of the problem is usually known; at the least, its dimension
often has an apparent upper bound. Motivated by such practical demands,
characterizations of positivity preserving functions have recently been
obtained in fixed dimensions, under further constraints that arise in
practice; see, for example,
\cite{GKR-lowrank, Guillot_Khare_Rajaratnam2012,
Guillot_Rajaratnam2012b}.

In the case of a fixed dimension $N$, obtaining characterizations of
functions which preserve matrix positivity when applied entrywise remains
a difficult open problem, even discouraging in view of the scarcity of
known results for low rank and low degree.
Using an idea of Loewner, Horn showed in \cite{horn}
that if a continuous function
$f : ( 0, \infty) \to \R$ satisfies
$f[ - ]: \bp_N( ( 0, \infty ) ) \to \bp_N( \R )$, then
$f \in C^{N - 3}( ( 0, \infty ) )$ and $f^{( k )}( x ) \geq 0$ for all
$x > 0$ and all $0 \leq k \leq N - 3$. Moreover, if it is known that
$f \in C^{N - 1}( ( 0, \infty ) )$, then $f^{( k )}( x ) \geq 0$ for
all $x > 0$ and all $0 \leq k \leq N - 1$.

As of today, obtaining an effective characterization of arbitrary
functions which preserve matrix positivity in a fixed dimension looks
rather inaccessible. The main result of the present article provides the
first such characterization for polynomial functions, and illustrates the
complexity of the coefficient bounds, even in the simplest of situations.
 
\begin{theorem}\label{Tthreshold}
Fix $\rho > 0$ and integers $N \geq 1$, $M \geq 0$ and let
$f( z ) = \sum_{j = 0}^{N - 1} c_j z^j + c' z^M$ be a
polynomial with real coefficients.
Also denote by $\overline{D}( 0, \rho )$ the closed disc in $\C$
with radius $\rho > 0$ and center the origin.
For any vector
$\bd := ( d_0, \ldots, d_{N - 1} )$ with non-zero entries, let
\begin{equation}\label{Erankone}
\mathcal{C}( \bd ) = \mathcal{C}( \bd; z^M; N, \rho ) := %
\sum_{j = 0}^{N - 1} \binom{M}{j}^2 \binom{M - j - 1}{N - j - 1}^2 %
\frac{\rho^{M - j}}{d_j}, 
\end{equation}
and let $\bc := ( c_0, \ldots, c_{N - 1} )$. The following are
equivalent.
\begin{enumerate}
\item $f[ - ]$ preserves positivity on
$\bp_N( \overline{D}( 0, \rho ) )$.

\item The coefficients $c_j$ satisfy either
$c_0$, \ldots, $c_{N - 1}$, $c' \geq 0$, or
$c_0$, \ldots, $c_{N - 1} > 0$ and
$c' \geq -\mathcal{C}( \bc )^{-1}$.

\item $f[ - ]$ preserves positivity on $\bp_N^1( ( 0, \rho ) )$, the
set of matrices in $\bp_N( ( 0, \rho ) )$ having rank at most~$1$.
\end{enumerate}
\end{theorem}

\noindent
Note that the necessity of having $c_0$, \ldots, $c_{N - 1} \geq 0$ in
part~(2) of the theorem follows from Horn's theorem as stated above.
The constant $\mathcal{C}( \bc; z^M; N, \rho )$ provides a threshold
for polynomials that preserve positivity on $\bp_N$ but not on
$\bp_{N + 1}$. Our theorem thus provides a quantitative version in
fixed dimension of Schoenberg's result, as well as of Horn's result.
As should be expected, our bound goes asymptotically to the
Schoenberg degree-free statement, since
$\mathcal{C}( \bc; z^M; N, \rho ) \to \infty$ as $N \to \infty$.
It is remarkable that the proof of Theorem~\ref{Tthreshold} is obtained
by using Schur polynomials.

\begin{remark}\label{Rreferee}
Notice that if $M < N$ then $\mathcal{C}( \bc; z^M; N, \rho )$ makes
sense and equals $c_M^{-1}$, if we use for any complex number
$z \in \C$ the formulas
\[
\binom{z}{n} := \frac{ z ( z - 1 ) \cdots ( z - n + 1 )}{n!} %
\quad \text{and} \quad \binom{z}{0} := 1.
\]
In effect, Theorem~\ref{Tthreshold} says in this case that when $f$ is a
polynomial of degree at most $N - 1$, the map $f[-]$ preserves
positivity on $\bp_N( \overline{D}( 0, \rho ) )$ if and only if all
coefficients of $f$ are non-negative.
\end{remark}

Theorem~\ref{Tthreshold} provides a decisive first step
towards isolating classes of functions that preserve positivity on
$\bp_N$ when applied entrywise. Additionally, the result yields a
wealth of interesting consequences that initiate the development of an
entrywise matrix calculus that leaves invariant the cone $\bp_N$,
in parallel and in contrast to the much better understood standard
functional calculus. The next theorem provides a constructive
criterion for preserving positivity, applicable to all analytic
functions. In the theorem and thereafter, the $N \times N$ matrix
with all entries equal to~$1$ is denoted by $\one{N}$, and
$A^{\circ k} := ( a_{i j}^k )$ denotes the $k$th Hadamard (or
entrywise, or Schur) power of $A$.

\begin{theorem}\label{Tanalytic}
Fix $\rho > 0$ and an integer $N \geq 1$. Let
$\bc := ( c_0, \ldots, c_{N - 1} ) \in ( 0, \infty )^N$, and
suppose $g( z ) := \sum_{M = N}^\infty c_M z^M$ is analytic on
$D( 0, \rho )$ and continuous on $\overline{D( 0, \rho )}$, with real
coefficients. Then
\begin{equation}
t ( c_0 {\bf 1}_{N \times N} + c_1 A + \cdots + %
c_{N - 1} A^{\circ ( N - 1 )} ) - g[ A ] \in \bp_N( \C )
\end{equation}
for all $A \in \bp_N( \overline{D}( 0, \rho ) )$ and all
\[
t \geq \sum_{M \geq N : c_M > 0} c_M \mathcal{C}( \bc; z^M; N, \rho).
\]
Moreover, this series is convergent, being bounded above
by
\begin{equation}\label{Eanalytic-bound}
\frac{g_2^{( 2 N - 2 )}( \sqrt{\rho} )}{2^{N - 1} ( N - 1)!^2} %
\sum_{j = 0}^{N - 1} \binom{N - 1}{j}^2 \frac{\rho^{N - j - 1}}{c_j},
\end{equation}
where $g_2( z ) := g_+( z^2 )$ and
\[
g_+( z ) := \sum_{M \geq N : c_M > 0} c_M z^M.
\]
\end{theorem}

\noindent Note that Theorem~\ref{Tthreshold} concerns the special case of
Theorem~\ref{Tanalytic} with $g(z) = c_M z^M$.

Theorem~\ref{Tanalytic} provides a sufficient condition for a large
class of functions to preserve positive semidefiniteness in fixed
dimension. The loosening of the tight thresholds for the individual
coefficients is compensated in this case by the closed form of the bound
for the lowest eigenvalue of the respective matrix pencil.

Next we describe some consequences of our main results. For
$A \in \bp_N( K )$ and $f$ as in Theorem~\ref{Tthreshold} with
$M \geq N$, note that
\[
f[ A ] = c_0 \one{N} + c_1 A + \cdots + %
c_{N - 1} A^{\circ ( N - 1 )} + c_M A^{\circ M},
\]
where $c_M = c'$.
Understanding when $f[ A ]$ is positive semidefinite is thus equivalent
to controlling the spectrum of linear combinations of Hadamard powers
of~$A$, by obtaining linear inequalities of the form
\begin{equation*}
c_0 \one{N} + c_1 A + \cdots + c_{N - 1} A^{\circ ( N - 1 )} + %
c_M A^{\circ M} \geq 0,
\end{equation*}
where the order is the Loewner ordering, given by the cone
$\bp_N( \C )$.  A direct application of our main theorem provides a
sharp bound for controlling the Hadamard powers of positive
semidefinite matrices.

\begin{corollary}\label{Clmi}
Fix $\rho > 0$, integers $M \geq N \geq 1$, and scalars $c_0$, \ldots,
$c_{N - 1} > 0$. Then
\begin{equation}\label{Elmi}
A^{\circ M} \leq \mathcal{C}( \bc; z^M; N, \rho ) \cdot \bigl( %
c_0 \one{N} + c_1 A + \cdots + c_{N-1} A^{\circ (N - 1)} \bigr)
\end{equation}
for all $A \in \bp_N( \overline{D}( 0, \rho) )$. Moreover, the
constant $\mathcal{C}( \bc; z^M; N, \rho )$ is sharp.
\end{corollary}

As an immediate consequence of Lemma~\ref{Lpos-coeff} below, notice
that the right-hand side of Equation~\eqref{Elmi} cannot be replaced
by a sum of fewer than $N$ Hadamard powers of~$A$.
Corollary~\ref{Clmi} thus yields a sharp bound for controlling the
Hadamard power $A^{\circ M}$ with the smallest number of powers of
lower order.

In a different direction, our main result naturally fits into the fast
developing area of spectrahedra \cite{SIAM-optimization, Vinzant}
and the matrix cube problem \cite{Nemirovski}.
The latter, a key technical ingredient in modern
optimization theory, continues to attract the attention of various groups
of researchers, mostly in applied mathematics.
Recall that given real
symmetric $N \times N$ matrices $A_0$, \ldots, $A_{M + 1}$, where
$M \geq 0$, the corresponding matrix cubes are
\begin{equation}\label{Ecube}
\mathcal{U}[\eta] := \Bigl\{ %
A_0 + \sum_{m = 1}^{M + 1} u_m A_m : u_m \in [ {-\eta}, \eta ] \Bigr\}
\qquad ( \eta > 0 ).
\end{equation}
The matrix cube problem consists of determining whether
$\mathcal{U}[ \eta ] \subset \bp_N$, and finding the largest $\eta$
for which this is the case. As another consequence of our main result,
we obtain an asymptotically sharp bound for the matrix cube problem
when the matrices $A_m$ are Hadamard powers.

\begin{corollary}\label{Ccube}
Fix $\rho > 0$ and integers $M \geq 0$, $N \geq 1$.
Given a matrix $A \in \bp_N( \overline{D}( 0, \rho ) )$, let
\[
A_0 := c_0 {\bf 1}_{N \times N} + c_1 A + \cdots + %
c_{N - 1} A^{\circ (N - 1)}
\]
and
\[
A_m := A^{\circ (N-1+m)} \qquad \textrm{for } 1 \leq m \leq M + 1,
\]
where the coefficients $c_0$, \ldots, $c_{N - 1} > 0$. Then
\begin{align}
\eta \leq \Biggl( \sum_{m = 0}^M %
\mathcal{C}(\bc; z^{N+m}; N, \rho ) \Biggr)^{-1} & \ \Rightarrow \ %
\mathcal{U}[ \eta ] \subset \bp_N( \C ) \\
& \ \Rightarrow \ \eta \leq %
\mathcal{C}( \bc; z^{N+M}; N, \rho )^{-1}.
\end{align}
The upper and lower bounds for $\eta$ are asymptotically equal as
$N \to \infty$, i.e.,
\begin{equation}\label{Easymptotic}
\lim_{N \to \infty} \mathcal{C}( \bc; z^{N+M}; N, \rho )^{-1} %
\sum_{m = 0}^M \mathcal{C}( \bc; z^{N+m}; N, \rho ) = 1.
\end{equation}
\end{corollary}

\noindent
See the end of Section~\ref{Sanalytic} for the proof of this result.

Finally, understanding which polynomials preserve positivity on
$\bp_N( K )$ can be reformulated as an extremal problem involving
generalized Rayleigh quotients of Hadamard powers.

\begin{theorem}\label{Crayleigh}
Fix $\rho > 0$, integers $M \geq N \geq 1$, and scalars $c_0$, \ldots,
$c_{N - 1} > 0$. Then
\begin{equation*}
\inf_{\bv \in \mathcal{K}( A )^\perp \setminus \{ 0 \}} %
\frac{\bv^* \Bigl( \sum_{j = 0}^{N - 1} c_j A^{\circ j} \Bigr) \bv}%
{\bv^* A^{\circ M} \bv} \geq \mathcal{C}( \bc; z^M; N, \rho )^{-1}
\end{equation*}
for all non-zero $A \in \bp_N( \overline{D}( 0, \rho ) )$, where
\[
\mathcal{K}( A ) := %
\ker( c_0 \one{N} + c_1 A + \cdots + c_{N - 1} A^{\circ ( N - 1 )} ).
\]
The lower bound $\mathcal{C}( \bc; z^M; N, \rho )^{-1}$ is sharp, and
may be obtained by considering only the set of rank-one matrices
$\bp_N^1( ( 0, \rho ) )$.
\end{theorem}

We note the surprising fact that the left-hand side of
the inequality above is \textit{not} continuous in the variable $A$;
see Remark~\ref{discts}. Motivated by this approach, we obtain, in
Theorem~\ref{Tsimult} below, a description of the kernel
$\mathcal{K}( A )$ for a given matrix $A \in \bp_N( \C )$. As we show,
this kernel coincides with the simultaneous kernels
$\cap_{n \geq 0} \ker A^{\circ n}$ of the Hadamard powers of $A$, and
leads to a
hitherto unexplored stratification
of the space $\bp_N(\C)$.

The rest of this paper is organized as follows. We review background
material in Section~\ref{Sback}. The main result of the paper is
proved in Section~\ref{Smain}, along with many intermediate results on
Schur polynomials that may be interesting in their own right. We also
show in Section~\ref{Smain} how our main result can naturally be
extended to general polynomials, and to analytic functions.
Section~\ref{Srayleigh} deals with the reformulation of our main
theorem as a variational problem, and provides a closed-form
expression for the extreme critical value of a single positive
semidefinite matrix.  Beginning with a novel block-matrix
decomposition of such matrices into rank-one components,
Section~\ref{Skernels} provides a description of the simultaneous
kernels of Hadamard powers of a positive semidefinite matrix. The last
section contains a unified presentation of a dozen known computations
in closed form, of extreme critical values for matrix pencils,
all relevant to our present work.

\subsection*{Acknowledgments}

The authors are grateful to Bala Rajaratnam for sharing his enthusiasm
and ideas that attracted the four of us to these topics. We would like
to thank the American Institute of Mathematics (AIM) for hosting the
workshop ``Positivity, graphical models, and modeling of complex
multivariate dependencies'' in October 2014, where this project was
initiated. We thank Yuan Xu for very stimulating
conversations at the AIM workshop, which were helpful in embarking upon
this project, and Shmuel Friedland, for valuable insights into
generalized Rayleigh quotients.
We are indebted to the referee for a careful examination of
the contents of the manuscript and constructive criticism on its
presentation.

\section{Background and notation}\label{Sback}

Given a subset $K \subset \C$ and integers $1 \leq k \leq N$, let
$\bp_N^k( K )$ denote the set of positive semidefinite $N \times N$
matrices with entries in $K$ and with rank at most $k$, and let
$\bp_N( K ) := \bp_N^N( K )$.
Given a matrix $A$, let $A^{\circ k}$ denote the matrix obtained from
$A$ by taking the $k$th power of each entry; in particular, if $A$ is
an $N \times N$ matrix then $A^{\circ 0} = \one{N}$, the $N \times N$
matrix with each entry equal to~$1$.

Recall that the \emph{Gegenbauer} or \emph{ultraspherical} polynomials
$C^{( \lambda )}_n( x )$ satisfy
\[
( 1 - 2 x t + t^2 )^{-\lambda} = %
\sum_{n = 0}^\infty C^{( \lambda )}_n( x ) t^n \qquad ( \lambda > 0 ),
\]
while for the \emph{Chebyshev polynomials of the first kind}
$C^{( 0 )}_n( x )$ we have
\[
( 1 - x t ) ( 1 - 2 x t + t^2 )^{-1} = %
\sum_{n = 0}^\infty C^{( 0 )}_n( x ) t^n.
\]

We begin by recalling Schoenberg's original statement, which
classifies positive definite functions on a sphere $S^{d - 1}$ of
fixed dimension.

\begin{theorem}[{Schoenberg,
\cite[Theorems~1 and~2]{Schoenberg42}}]\label{Tspheres}
Fix an integer $d \geq 2$ and a continuous function
$f : [ {-1}, 1 ] \to \R$.
\begin{enumerate}
\item The function $f( \cos \cdot )$ is positive definite on the unit
sphere $S^{d  - 1}$ if and only if $f$ can be written as a
non-negative linear combination of the Gegenbauer or Chebyshev
polynomials $C_n^{( \lambda )}$, where $\lambda = ( d - 2 ) / 2$:
\[
f( x ) = %
\sum_{n \geq 0} a_n C_n^{( \lambda )}( x ) \qquad ( a_n \geq 0 ).
\]

\item The entrywise function
$f[ - ] : \bp_N( [ {-1}, 1 ] ) \to \bp_N( \R )$ for all $N \geq 1$ if
and only if $f$ is analytic on $[ {-1}, 1 ]$ and absolutely monotonic
on $[ 0, 1 ]$, i.e., $f$ has a Taylor series with non-negative
coefficients convergent on the closed unit disc
$\overline{D}( 0, 1 )$.
\end{enumerate}
\end{theorem}

For more on absolutely monotonic functions, see the work
\cite{Bernstein} of Bernstein.
Rudin \cite{Rudin59} proved part (2) of the above result
without the continuity assumption, in addition to several other
characterizations of such a function $f$. His work was a part of the
broader context of studying functions acting on Fourier--Stieltjes
transforms in locally compact groups, as explored in joint works with
Kahane, Helson, and Katznelson in \cite{HKKR,Kahane-Rudin}.

The work of Schoenberg has subsequently been extended along
several directions; see, for example,
\cite{BCR-semigroups, Bochner-pd, Bochner-zonal, Christensen_et_al78,
Herz63, Hiai2009, vonNeumann-Schoenberg}.
However, the solution to the original problem in fixed dimension
remains elusive when $N > 2$.

An interesting necessary condition for a continuous function
$f : ( 0, \infty) \to \R$ to preserve positivity in fixed dimension
has been provided by Horn \cite{horn}. This result was recently
extended in \cite{GKR-lowrank} to apply in the case of low-rank
matrices with entries in $( 0, \rho )$ for some $\rho > 0$, and
without the continuity assumption; on the last point, see also Hiai's
work \cite{Hiai2009}.

\begin{theorem}[{Horn \cite{horn}, Guillot--Khare--Rajaratnam
\cite{GKR-lowrank}}]\label{Thorn}
Suppose $f : I \to \R$, where $I := ( 0, \rho )$ and
$0 < \rho \leq \infty$. Fix an integer $N \geq 2$ and suppose that
$f[ A ] \in \bp_N( \R )$ for any $A \in \bp_N^2( I )$ of the form
$A = a \one{N} + \bu \bu^T$, where $a \in ( 0, \rho )$ and
$\bu \in [ 0, \sqrt{\rho - a} )^N$. Then $f \in C^{N - 3}(I)$, with
\[
f^{( k )}( x ) \geq 0 \qquad \forall x \in I, \ 0 \leq k \leq N - 3,
\]
and $f^{(N - 3)}$ is a convex non-decreasing function on~$I$.
If, further, $f \in C^{N - 1}( I )$, then $f^{( k )}( x ) \geq 0$ for
all $x \in I$ and $0 \leq k \leq N - 1$.
\end{theorem}

Note that all real power functions preserve positivity on
$\bp_N^1( ( 0, \rho ) )$, yet such functions need not have even a
single positive derivative on $( 0, \rho )$. However, as shown in
Theorem~\ref{Thorn}, working with a small one-parameter extension of
$\bp_N^1( ( 0, \rho ) )$ guarantees that $f^{( k )}$ is non-negative
on $( 0, \rho )$ for $0 \leq k \leq N - 3$.

\begin{remark}\label{Rsharp}
Theorem~\ref{Thorn} is sharp in the sense that there exist functions
$f: ( 0, \rho ) \to \R$ which preserve positivity on
$\bp_N( ( 0, \rho ) )$, but not on $\bp_{N + 1}( ( 0, \rho ) )$. For
example, $f( x ) = x^\alpha$, where $\alpha \in ( N - 2, N - 1 )$ is
an example of such a function; see \cite{FitzHorn, GKR-crit-2sided,
Hiai2009} for more details. Thus the bound on the number of
non-negative derivatives in Theorem~\ref{Thorn} is sharp.
\end{remark}

In light of Remark~\ref{Rsharp}, we focus henceforth on analytic
functions which preserve $\bp_N( K )$ for fixed $N$ when applied
entrywise. Note that any analytic function on $D( 0, \rho )$ that maps
$( 0, \rho )$ to $\R$ necessarily has real Taylor coefficients.

Recall by Theorem~\ref{Thorn} that if
$f[ - ] : \bp_N^2( ( 0, \infty ) ) \to \bp_N( \R )$ and
$f \in C^{(N - 1)}( ( 0, \infty ) )$ then $f^{( k )}$ is non-negative
on $( 0, \infty )$ for $0 \leq k \leq N - 1$. The next lemma shows
that if $f$ is assumed to be analytic, then it suffices to work with
$\bp_N^1$ instead of $\bp_N^2$ in order to arrive at the same
conclusion.

\begin{lemma}\label{Lpos-coeff}
Let $f : D( 0, \rho ) \to \R$ be an analytic function, where
$0 < \rho \leq \infty$, so that $f( x ) = \sum_{n \geq 0} c_n x^n$ on
$D( 0, \rho )$. If $f[ - ] : \bp_N^1( ( 0, \rho ) ) \to \bp_N( \R )$
for some integer $N \geq 1$, then the first $N$ non-zero Taylor
coefficients $c_j$ are strictly positive.
\end{lemma}

In particular, if $f$ is a sum of at most $N$ monomials, then under the
hypotheses of Lemma~\ref{Lpos-coeff}, all coefficients of $f$ are
non-negative.

\begin{proof}
Suppose that the first $m$ non-zero coefficients are $c_{n_1}$,
\ldots, $c_{n_m}$, where $m \leq N$. Fix
$\bu := ( u_1, \ldots, u_m )^T \in ( 0, \sqrt{\rho} )^m$ with distinct
entries $u_1$, \ldots, $u_m$, and note the matrix
$( u_j^{n_k} )_{j, k = 1}^m$ is non-singular
\cite[Chapter~XIII, \S8, Example~1]{Gantmacher_Vol2}. Hence
$\{ \bu^{\circ n_1}, \ldots, \bu^{\circ n_m} \}$ is linearly
independent and there exist vectors $\bv_1$, \ldots, $\bv_m \in \R^m$
such that $( \bu^{\circ n_j} )^T \bv_k = \delta_{j, k}$ for all
$1 \leq j$, $k \leq m$. Since $f$ preserves positivity on
$\bp_N^1( ( 0, \rho ) )$, it follows that
\[
0 \leq \epsilon^{-n_k} \bv_k^T f[ \epsilon \bu \bu^T ] \bv_k = %
c_{n_k} + \sum_{j > n_m} c_j ( \bu^{\circ j} \bv_k )^2 %
\epsilon^{j  - n_k} \to c_{n_k}
\]
as $\epsilon \to 0^+$, for $1 \leq k \leq m$. The result follows.
\end{proof}

The above discussion naturally raises various questions.

\begin{enumerate}
\item Can one find necessary and sufficient conditions on a restricted
class of functions, such as polynomials, to ensure that positivity is
preserved in fixed dimension?

\item Note that the power functions in Remark~\ref{Rsharp},
$f( x ) = x^\alpha$ for $\alpha \in ( N - 2, N - 1 )$, are not
analytic. Does there exist a function $f$ analytic on an open subset
$U \subset \C$ which preserves positivity on $\bp_N( U )$, but not on
$\bp_{N + 1}( U )$?
\end{enumerate}

\noindent
Our main result provides positive answers to both of these questions.

\section{Schur polynomials and Hadamard powers}\label{Smain}

The goal of this section is to prove Theorem~\ref{Tthreshold}. The
proof relies on a careful analysis of the polynomial
\[
p( t ) := \det( t ( c_0 \one{N} + c_1 A + \cdots + %
c_{N - 1} A^{\circ ( N - 1 )} ) - A^{\circ M} ),
\]
where $A \in \bp_N^1( \overline{D}( 0, \rho) )$. More specifically, we
study algebraic properties of the polynomial $p( t )$, and show how an
explicit factorization can be obtained by exploiting the theory of
symmetric polynomials.

\subsection{Determinantal identities for Hadamard powers}

We begin with some technical preliminaries involving Schur polynomials.

As the results in this subsection may be of independent interest to
specialists in symmetric functions and algebraic combinatorics, we
state them over an arbitrary field $\F$.

Given a partition, i.e., a non-increasing $N$-tuple of
non-negative integers $\bn = ( n_N \geq \cdots \geq n_1 )$, the
corresponding \emph{Schur polynomial} $s_\bn( x_1, \ldots, x_N )$ over a
field $\F$ with at least $N$ elements is defined to be the unique
polynomial extension to $\F^N$ of
\begin{equation}\label{schurdef}
s_\bn( x_1, \ldots, x_N ) := %
\frac{\det ( x_i^{n_j + N - j} )}{\det ( x_i^{N - j} )}
\end{equation}
for pairwise distinct $x_i \in \F$. Note that the denominator
is precisely the Vandermonde determinant
$\Delta_N( x_1, \ldots, x_N ) := \det ( x_i^{N - j} ) = %
\prod_{1 \leq i < j \leq N} ( x_i - x_j )$; it follows from this that
\begin{equation}\label{schureval}
s_\bn( 1, \ldots, z^{N - 1} ) = %
\prod_{1 \leq i < j \leq N} %
\frac{z^{n_j + j} - z^{n_i + i}}{z^j - z^i}, \quad %
s_\bn( 1, \ldots, 1 ) = \prod_{1 \leq i < j \leq N} %
\frac{n_j - n_i + j - i}{j - i}.
\end{equation}
The last equation can also be deduced from Weyl Character Formula
in type~$A$; see, for example,
\cite[Chapter~I.3, Example~1]{Macdonald}. For more details about Schur
polynomials and the theory of symmetric functions, see
\cite{Macdonald}. In particular, note that Schur polynomials have
non-negative integer coefficients, by
\cite[Chapter~I, Equation~(5.12)]{Macdonald}.

\begin{proposition}\label{Pcauchy-binet}
Let $A := \bu \bv^T$, where $\bu = ( u_1, \ldots, u_N )^T$ and
$\bv := ( v_1, \ldots, v_N )^T \in \F^N$ for $N \geq 1$.  Given
$m$-tuples of non-negative integers
$\bn = ( n_m > n_{m-1} > \cdots > n_1 )$
and scalars $( c_{n_1}, \ldots, c_{n_m} ) \in \F^m$,
the following determinantal identity holds:
\begin{equation}\label{Ecauchy-binet1}
\det \sum_{j = 1}^m c_{n_j} A^{\circ n_j} = %
\Delta_N( \bu ) \Delta_N( \bv ) \sum_{\bn' \subset \bn, \ | \bn'| = N}
s_{\lambda( \bn' )}( \bu ) s_{\lambda( \bn' )}( \bv ) %
\prod_{k = 1}^N c_{n'_k}.
\end{equation}
Here,
$\lambda( \bn') := %
( n'_N - N + 1 \geq n'_{N-1} - N + 2 \geq \cdots \geq n'_1 )$
is obtained by subtracting the staircase partition
$( N - 1, \ldots, 0 )$ from $\bn' := ( n'_N > \cdots > n'_1 )$,
and the sum is over all subsets $\bn'$ of cardinality~$N$. In
particular, if $m < N$ then the determinant is zero.
\end{proposition}
\begin{proof}
If there are $m < N$ summands then the matrix in question has rank at
most $m < N$, so it is singular; henceforth we suppose $m \geq N$.
Note first that if $\bc := ( c_{n_1}, \ldots, c_{n_m} )$ and
\[
X( \bu, \bn, \bc ) := %
( \sqrt{c_{n_k}} u_j^{n_k} )_{1 \leq j \leq N, 1 \leq k \leq m}
\]
where we work over an algebraic closure of $\F$, then
\begin{equation}\label{Ecauchy-binet2}
\sum_{j = 1}^m c_{n_j} A^{\circ n_j} = %
X( \bu, \bn, \bc ) X( \bv, \bn, \bc )^T.
\end{equation}
Next, let $\bc|_{\bn'} := ( c_{n'_1}, \ldots, c_{n'_N} )$ and note
that, by the Cauchy--Binet formula applied to~\eqref{Ecauchy-binet2},
\begin{align*} 
\det \sum_{j = 1}^m c_{n_j} A^{\circ n_j} & = %
\sum_{\bn' \subset \bn, \ | \bn' | = N} %
\det \bigl( X( \bu, \bn', \bc|_{\bn'} ) %
X( \bv, \bn', \bc|_{\bn'} )^T \bigr) \\
 & = \sum_{\bn' \subset \bn, \ | \bn' | = N} %
\det X( \bu, \bn', \bc|_{\bn'} ) \det X( \bv, \bn', \bc|_{\bn'} ) \\
 & = \sum_{\bn' \subset \bn, \ | \bn' | = N} %
\det ( u_j^{n'_k} ) \det ( v_j^{n'_k} ) \prod_{k = 1}^N c_{n'_k}.
\end{align*}
Each of the last two determinants is precisely the product of the
appropriate Vandermonde determinant times the Schur polynomial
corresponding to $\lambda( \bn' )$; see Equation~\eqref{schurdef}.
This observation completes the proof.
\end{proof}

We now use Proposition~\ref{Pcauchy-binet} to obtain an explicit
factorization of the determinant of $p[ \bu \bv^T ]$ for a large class
of polynomials~$p$.

\begin{theorem}\label{Pjacobi-trudi}
Let $R \geq 0$ and $M \geq N \geq 1$ be integers, let $c_0$, \ldots,
$c_{N - 1} \in \F^\times$ be non-zero scalars, and let the polynomial
\[
p_t( x ) := t ( c_0 x^R + \cdots + c_{N - 1} x^{R + N - 1} ) - %
x^{R + M},
\]
where $t$ is a variable. Let the hook partition
$\mu( M, N, j ) := ( M - N + 1, 1, \ldots, 1, 0, \ldots, 0 )$, with
$N - j - 1$ entries after the first equal to $1$ and the remaining $j$
entries equal to $0$. The following identity holds for all
$\bu = ( u_1, \ldots, u_N )$ and
$\bv := ( v_1, \ldots, v_N ) \in \F^N$:
\begin{equation}\label{Ejacobi-trudi}
\det p_t[ \bu \bv^T ] = t^{N - 1} \Delta_N( \bu ) %
\Delta_N( \bv ) \prod_{j = 1}^N c_{j - 1} u_j^R v_j^R %
\Bigl( t - \sum_{j = 0}^{N - 1} %
\frac{s_{\mu( M, N, j )}( \bu ) %
s_{\mu( M, N, j )}( \bv )}{c_j} \Bigr).
\end{equation}
\end{theorem}

\begin{proof}
Let $A = \bu \bv^T$ and note first that 
$\det ( A^{\circ R} \circ B ) = \prod_{j = 1}^N u_j^R v_j^R \cdot \det B$
for any $N \times N$ matrix $B$, so it suffices to prove the result when
$R = 0$, which we assume from now on.

Recall the Laplace formula: if $B$ and $C$ are $N \times N$ matrices,
then
\begin{equation}\label{Edet_sum}
\det( B + C ) = \sum_{\bn \subset \{ 1, \ldots, N \}} \det M_\bn( B; C ),
\end{equation}
where $M_\bn( B; C )$ is the matrix formed by replacing the rows of
$B$ labelled by elements of $\bn$ with the corresponding rows of $C$.
In particular, if $B = \sum_{j = 0}^{N - 1} c_j A^{\circ j}$ then
\begin{equation}\label{heqn}
\det p_t[ A ] = \det( t B - A^{\circ M} ) = t^N \det B - %
t^{N - 1} \sum_{j = 1}^N \det M_{\{ j \}}( B; A^{\circ M} ),
\end{equation}
since the determinants in the remaining terms contain two rows of the
rank-one matrix $A^{\circ M}$. By Proposition~\ref{Pcauchy-binet}
applied with $n_j = j - 1$, we obtain
\begin{align*}
\det B = \Delta_N( \bu ) \Delta_N( \bv ) c_0 \cdots c_{N - 1}.
\end{align*}
To compute the coefficient of $t^{N-1}$, note that taking $t = 1$ in
Equation~\eqref{heqn} gives that
\[
\sum_{j = 1}^N \det M_{\{ j \}}( B; A^{\circ M} ) = %
\det B - \det p_1[ A ].
\]
Moreover, $\det p_1[ A ]$ can be computed using
Proposition~\ref{Pcauchy-binet} with $m = N + 1$ and
$c_{n_{N + 1}} = {-1}$:
\begin{equation*}
\det p_1[ A ] = \det B - \Delta_N( \bu ) \Delta_N( \bv ) %
c_0 \cdots c_{N - 1} \sum_{j = 0}^{N - 1} %
\frac{s_{\mu( M, N, j )}( \bu ) s_{\mu( M, N, j)}( \bv )}{c_j},
\end{equation*}
since
$\mu( M, N, j ) = %
\lambda\bigl( ( M, N - 1, N - 2, \ldots, j + 1, %
\widehat{j}, j - 1, \ldots, 0 ) \bigr)$ for $0 \leq j \leq N - 1$. The
identity~\eqref{Ejacobi-trudi} now follows.
\end{proof}

\begin{remark}
The connection between Theorem~\ref{Pjacobi-trudi} and the constant
$\mathcal{C}( \bc; z^M; N, \rho )$ in Theorem~\ref{Tthreshold} stems
from the fact that
\begin{equation}\label{Eschur1}
s_{\mu( M, N, j )}( 1, \ldots, 1 ) = %
\binom{M}{j} \binom{M - j - 1}{N - j - 1}
\end{equation}
for $0 \leq j \leq N - 1$. 
This is a straightforward application of Equation~\eqref{schureval};
alternatively, it follows by applying Stanley's hook-content formula
\cite[Theorem 15.3]{Stanley} to the hook Schur function
$\mu( M, N, j )$. We also mention a third proof using the dual
Jacobi--Trudi (Von N{\"a}gelsbach--Kostka) identity
\cite[Chapter~I, Equation~(3.5)]{Macdonald}, which rewrites Schur
polynomials in terms of elementary symmetric polynomials. The proof
goes as follows: note that the dual partition
of $\mu( M, N, j )$ is, up to attaching zeros at the end,
the $(M + 1)$-tuple $\mu'( M, N, j ) := ( N - j, 1, \ldots, 1 )$.
Therefore $s_{\mu( M, N, j )}( 1, \ldots, 1 )$ equals the determinant
of a $2 \times 2$ block triangular matrix, one of whose block diagonal
submatrices is unipotent, and the determinant of the other is computed
inductively to yield~\eqref{Eschur1}.
\end{remark}

\begin{corollary}
In the setting of Theorem~\ref{Pjacobi-trudi}, if $M = N$, then 
\begin{equation}
\det p_t[ A ] = t^{N - 1} \Delta_N( \bu ) \Delta_N( \bv ) %
\prod_{j = 1}^N c_{j - 1} u_j^R v_j^R \Bigl( t - %
\sum_{j = 0}^{N - 1} \frac{e_{N - j}( \bu ) e_{N - j}( \bv )}{c_j} %
\Bigr),
\end{equation}
where
$e_j( \bu ) := %
\sum_{1 \leq n_1 < \cdots < n_j \leq N} u_{n_1} \ldots u_{n_j}$
is the elementary symmetric polynomial in~$N$ variables of degree~$j$.
\end{corollary}

\begin{proof}
This follows immediately from the observation that
$s_{\mu( 0, N, j )}$ is equal to $e_{N - j}$; see
\cite[Chapter~I, Equation~(3.9)]{Macdonald}.
\end{proof}

We conclude this part with an identity, which shows how the Schur
polynomials $s_{\mu( M, N, j )}$ in Theorem~\ref{Pjacobi-trudi} can be
used to express the Hadamard power $A^{\circ M}$ as a combination of
lower Hadamard powers.

\begin{lemma}\label{Lpower_exp}
Fix integers $M \geq N \geq 1$, and let the $N \times N$ matrix $A$
have entries in~$\F$. Denote the rows of $A$ by $\ba_1$, \ldots,
$\ba_N$. Then
\begin{equation}\label{Emiracle2}
A^{\circ M} = \sum_{j = 0}^{N - 1} D_{M, j}( A ) A^{\circ j},
\end{equation}
where $D_{M, j}( A )$ is the diagonal matrix
\[
( {-1} )^{N - j - 1} \diag\bigl( s_{\mu( M, N, j )}( \ba_1 ), \ldots, 
s_{\mu( M, N, j )}( \ba_N ) \bigr),
\]
and $s_{\mu( M, N, j)}$ is as in Theorem~\ref{Pjacobi-trudi}.
\end{lemma}

\begin{proof}
Let $\bv := ( v_1, \ldots, v_N )$, where $v_1$, \ldots, $v_N$ are
pairwise distinct and transcendental over~$\F$; we now work in the
field $\F( v_1, \ldots, v_n )$. If $V$ is the Vandermonde matrix
$( v_i^{j - 1} )$ then, by Cramer's rule, the solution to the equation
\[
V \bu = ( \bv^{\circ M} )^T \quad \iff \quad %
\bv^{\circ M} = \sum_{j = 0}^{N - 1} u_{j + 1} \bv^{\circ j}
\]
is given by setting $u_i$ to equal the Schur polynomial
$s_{\mu( M, N, i - 1 )}( \bv )$ times $( {-1} )^{N - i}$, to account
for transposing the $i$th column to the final place. The result now
follows from this identity, applied by specializing $\bv$ to each row
of $A$.
\end{proof}

\subsection{Proof of the main theorem}

Using the technical results on Schur polynomials established above, we
can now prove Theorem~\ref{Tthreshold}.

\begin{proof}[Proof of Theorem~\ref{Tthreshold}]
We begin by proving the result when $0 \leq M \leq N - 1$. In this
case, it follows immediately from Lemma~\ref{Lpos-coeff} that the
theorem holds,
since $\mathcal{C}( \bc; z^M; N, \rho )$ $= c_M^{-1}$ by
Remark~\ref{Rreferee}.

Now assume $M \geq N$. To be consistent with the statement of the
theorem, where the floating coefficient is denoted by $c'$, for
$M \geq N$ we adopt the unifying notation $c_M = c'$. Clearly (1)
implies (3). We will now show that (3) implies (2), and (2) implies
(1).

\noindent$\mathbf{(3) \implies (2).}$
Suppose that $f[ - ]$ preserves positivity on
$\bp_N^1( ( 0, \rho ) )$. By Lemma~\ref{Lpos-coeff}, the first $N$
non-zero coefficients of $f$ are positive, so we suppose $c_M < 0$ and
prove that $c_M \geq {-\mathcal{C}( \bc; z^M; N, \rho   )^{-1}}$.
Define $p_t( x )$ as in Theorem~\ref{Pjacobi-trudi}, with $R = 0$, and
set $t := | c_M |^{-1}$, so that $| c_M |^{-1} f[ - ] = p_t[ - ]$
preserves positivity on rank-one matrices $A = \bu \bu^T$ with
$\bu \in ( 0, \sqrt{\rho} )^n$. Then, by
Equation~\eqref{Ejacobi-trudi},
\begin{equation}\label{Enecessary}
0 \leq \det p_t[ \bu \bu^T ] = t^{N - 1} \Delta_N( \bu )^2 %
c_0 \cdots c_{N - 1} \Bigl( t - \sum_{j = 0}^{N - 1} %
\frac{s_{\mu( M, N, j )}( \bu )^2}{c_j} \Bigr).
\end{equation}
Now set $u_k := \sqrt{\rho} ( 1 - t' \epsilon_k )$, with
$\epsilon_k \in ( 0, 1 )$ pairwise distinct and $t' \in ( 0, 1 )$, so
that $\Delta_N( \bu ) \neq 0$. Taking the limit as $t' \to 0$, since
the final term in Equation~\eqref{Enecessary} must be non-negative, we
conclude that
\begin{align}\label{Elowerbound}
t = |c_M|^{-1} \geq \sum_{j = 0}^{N - 1} %
\frac{s_{\mu( M, N, j )}( \sqrt{\rho}, \ldots, \sqrt{\rho} )^2}{c_j} %
 & = \sum_{j = 0}^{N - 1} s_{\mu( M, N, j )}( 1, \ldots, 1 )^2 %
\frac{\rho^{M - j}}{c_j} \notag \\
 & = \mathcal{C}( \bc; z^M; N, \rho ),
\end{align}
as claimed.

\noindent$\mathbf{(2) \implies (1).}$
Suppose (2) holds. Note that (1) follows if $c_j \geq 0$ for all $j$, by
the Schur product theorem, so we assume that $c_0, \ldots, c_{N - 1} > 0$
and $-\mathcal{C}( \bc; z^M; N, \rho )^{-1} \leq c_M < 0$.

We first show that $f[ - ]$ preserves positivity on
$\bp_N^1(\overline{D}( 0, \rho ) )$. For $1 \leq m \leq N$, let
\begin{equation}
C_m := \sum_{j = 0}^{m - 1} s_{\mu( M-N+m, m, j )}( 1, \ldots, 1 )^2 %
\frac{\rho^{m + M-N - j}}{c_{N - m + j}}
= \mathcal{C}( \bc_m; z^{M - N + m}; m, \rho ),
\end{equation}
where $\bc_m := ( c_{N - m}, \ldots, c_{N - 1} )$. Note in this case
that $C_N$ is precisely $\mathcal{C}( \bc; z^M; N, \rho )$ as defined
in Equation~\eqref{Erankone} (and $N$, $M$ are fixed). It follows from
Theorem~\ref{Pjacobi-trudi} that $C_1 = \rho^{M - N + 1} / c_{N - 1}$
and, for $1 \leq m \leq N - 1$,
\begin{align*}
&C_{m + 1} - C_m \\
 & \geq \sum_{j = 0}^{m - 1} \bigl( %
s_{\mu( M - N + m + 1, m + 1, j + 1 )}( 1, \ldots, 1 )^2 - %
s_{\mu( M - N + m + 1, m, j )}( 1, \ldots, 1 )^2 \bigr) %
\frac{\rho^{m + M - N - j}}{c_{N - m + j}} \\
 & > 0,
\end{align*}
since
\begin{align*}
\frac{s_{\mu( M - N + m + 1, m + 1, j + 1 )}( 1, \ldots, 1 )}%
{s_{\mu( M - N + m, m, j )}( 1, \ldots, 1 )} & = %
\binom{m + 1 + M - N}{j + 1} \binom{m + M - N}{j}^{-1} \\
 & = \frac{m + 1 + M - N}{j + 1} > 1.
\end{align*}
Thus $0 < C_1 < C_2 < \cdots < C_N$.

Next, we claim that for $1 \leq m \leq N$ and all
$A = \bu \bu^* \in \bp_N^1( \overline{D}( 0, \rho ) )$, every
principal $m \times m$ submatrix of the matrix
\begin{equation}\label{Ematrix}
L := C_m ( c_{N - m} \one{N} + c_{N - m + 1} A + \cdots + %
c_{N - 1} A^{\circ ( m - 1 )} ) - A^{\circ ( m + M - N )}
\end{equation}
is positive semidefinite; for $m = N$, this gives immediately the
rank-one case of~(1).

The proof of the claim is by induction. The case $m = 1$ is immediate,
since a general diagonal entry of~$L$ for $m = 1$ equals
$C_1 c_{N - 1} - a_{j j}^{M - N + 1} = %
\rho^{M - N + 1} - a_{j j}^{M - N + 1} \geq 0$.
Now suppose the result holds for $m - 1 \geq 1$. In the remainder of
this proof, we adopt the following notation: given a non-empty set
$\bn \subset \{ 1, \ldots, N \}$ and an $N \times N$ matrix $B$,
denote by $B_\bn$ the principal submatrix of $B$ consisting of those
rows and columns labelled by elements of $J$; we adopt a similar
convention for the subvector $\bu_\bn$ of a vector $\bu$. If
$\bn \subset \{ 1, \ldots, N \}$ has cardinality $m$ then, by
Theorem~\ref{Pjacobi-trudi} with $\bv = \overline{\bu_\bn}$,
\[
\det L_\bn = %
C_m^{m - 1} | \Delta_m( \bu_\bn )|^2 \prod_{j = 1}^m c_{N - j} %
\Bigl( C_m - \sum_{j = 0}^{m - 1} %
\frac{| s_{\mu( M - N + m, m, j )}( \bu_\bn ) |^2}{c_{N - m + j}} \Bigr).
\]
Using the triangle inequality in $\C$ and the fact that the
coefficients of any Schur polynomial are non-negative, it follows
immediately that
\begin{align}\label{Etriangle}
| s_{\mu( M - N + m, m, j )}( \bu_\bn ) |^2 & \leq %
s_{\mu( M - N + m, m, j )}( \sqrt{\rho}, \ldots, \sqrt{\rho} )^2 %
\notag \\
 & = %
s_{\mu( M - N + m, m, j )}( 1, \ldots, 1 )^2 \rho^{m + M - N - j},
\end{align}
and so, by the choice of $C_m$, the determinant $\det L_\bn \geq 0$
for any set $\bn \subset \{ 1, \ldots, N \}$ of cardinality $m$.
Additionally, for each non-empty subset $\bn$ of cardinality $k < m$,
\begin{align*}
L_\bn & \geq C_m ( c_{N - k} A_\bn^{\circ ( m - k )} + \cdots + %
c_{N - 1} A_\bn^{\circ ( m - 1 )} ) - A_\bn^{\circ ( m + M - N)} \\
 & \geq A_\bn^{\circ ( m - k )} \circ \bigl( C_k %
( c_{N - k} \one{k} + \cdots + c_{N - 1} A_\bn^{\circ ( k - 1 )} ) %
- A_\bn^{\circ ( k + M - N )} \bigr),
\end{align*}
since $C_m > C_k$. It follows by the Schur product theorem and the
induction hypothesis that $\det L_\bn \geq 0$ for any non-empty subset
$\bn$ of cardinality $k \leq m$. Hence all principal $m \times m$
submatrices of $L$ are positive semidefinite, which concludes the
proof of the claim by induction. In particular, $f[ - ]$ preserves
positivity on $\bp_N^1( \overline{D}( 0, \rho ) )$.

We now prove that (1) holds for matrices of any rank; again we proceed
by induction. For $N = 1$, the result holds because
$\bp_1( \overline{D}( 0, \rho ) ) = %
\bp_1^1( \overline{D}( 0, \rho ) )$. Now suppose (1) holds for
$N - 1 \geq 1$, and let
\[
p_t[ B; M, \bd ] := t ( d_0 {\bf 1} + d_1 B + \cdots + %
d_{n - 1} B^{\circ ( n - 1 )}) - B^{\circ ( n + M )},
\]
for any square matrix $B$ of arbitrary order,
where $t$ is a real scalar and $n$ is the
length of the tuple $\bd = ( d_0, \ldots, d_{n - 1} )$. It suffices to
show that $p_t[ A; M - N, \bc ] \geq 0$ for
all~$t \geq \mathcal{C}( \bc; z^M; N, \rho )$ and all
$A = ( a_{i j} ) \in \bp_N( \overline{D}( 0, \rho ) )$.

By a lemma of FitzGerald and Horn \cite[Lemma~2.1]{FitzHorn}, if
$\bu \in \C^N$ is defined to equal
$( a_{i N} / \sqrt{a_{N N}} )_{i = 1}^N$ when $a_{N N} \neq 0$,
and is the zero vector otherwise, then $A - \bu \bu^*$ is positive
semidefinite and its final column and row are both zero. Furthermore,
the principal minor
\[
\Bigl| \begin{array}{cc} a_{i i} & a_{i N} \\
\overline{a_{i N}} & a_{N N} \end{array} \Bigr| = %
a_{i i} a_{N N} - | a_{i N} |^2 \ge 0,
\]
so the entries of $\bu \bu^*$ lie in $\overline{D}( 0, \rho )$. Since
\[
f( z ) - f( w ) = %
\int_0^1 ( z - w ) f'( \lambda z + ( 1 - \lambda ) w ) \std \lambda %
\qquad ( z, w \in \C )
\]
for any entire function $f$, it follows that
\begin{equation}\label{eqn:convex}
p_t[ A; M - N, \bc ] = p_t[ \bu \bu^*; M - N, \bc ]
+ \int_0^1 ( A - \bu \bu^* ) \circ %
M p_{t / M}[ \lambda A + ( 1 - \lambda ) \bu \bu^*; M - N,
\bc' ] \std \lambda,
\end{equation}
where the $(N - 1)$-tuple
$\bc' := (c _1, 2 c_2, \ldots, ( N - 1 ) c_{N - 1} )$. Now, since
$A - \bu \bu^*$ has last row and column both zero, the integrand
in~\eqref{eqn:convex} is positive semidefinite if the matrix
$p_{t / M}[ A_\lambda; M - N, \bc' ]$ is, where
$A_\lambda \in \bp_{N - 1}( \overline{D}( 0, \rho ) )$ is obtained by
deleting the final row and column of
$\lambda A + ( 1 - \lambda ) \bu \bu^*$. Thus, if we show the
inequality
\begin{equation}\label{Eineq}
\mathcal{C}( \bc; z^M; N, \rho ) \geq %
M \cdot \mathcal{C}( \bc'; z^{M - 1}; N - 1, \rho ),
\end{equation}
then both terms in Equation~\eqref{eqn:convex} are positive
semidefinite for all $t \geq \mathcal{C}( \bc; z^M; N, \rho )$,
by the induction hypothesis, and this gives the result.

Finally, we prove \eqref{Eineq}: note that
\begin{align*}
M \cdot \mathcal{C}( \bc'; z^{M-1}; N-1, \rho ) & = %
M \sum_{j = 0}^{N - 2} \binom{M - 1}{j}^2 %
\binom{M - 1 - j - 1}{N - 1 - j - 1}^2 %
\frac{\rho^{M - 1 - j}}{( j + 1 ) c_{j + 1}} \\
& = \sum_{j = 0}^{N - 2} \binom{M}{j + 1}^2 %
\binom{M - ( j + 1 ) - 1}{N - ( j + 1 ) - 1}^2 %
\frac{\rho^{M - ( j + 1 )}}{c_{j + 1}} \cdot \frac{j + 1}{M} \\
& = %
\sum_{j = 1}^{N - 1} \binom{M}{j}^2 \binom{M - j - 1}{N - j - 1}^2 %
\frac{\rho^{M - j}}{c_j} \cdot \frac{j}{M} \\
& \leq \mathcal{C}( \bc; z^M; N, \rho ).
\end{align*}
Thus (2) implies (1), and so concludes the proof of the theorem.
\end{proof}

\begin{remark}
Theorem~\ref{Tthreshold} shows further that,
for any subset $K \subset \C$ satisfying
\[
( 0, \rho ) \subset K \subset \overline{D}( 0, \rho),
\]
the polynomial
$f( z ) := %
c_0 + c_1 z + \cdots + c_{N - 1} z^{N - 1} + c_M z^M$
preserves positivity on $\bp_N( K )$ if and only if
$( c_0, \ldots, c_{N - 1}, c_M )$ satisfies
Theorem~\ref{Tthreshold}(2). More generally, given any class $\bp$ of
matrices such that
$\bp_N^1( ( 0, \rho ) ) \subset \bp \subset %
\bp_N( \overline{D}( 0, \rho ) )$,
Theorem~\ref{Tthreshold} implies that $f[ - ]$ preserves positivity on
$\bp$ if and only if $( c_0, \ldots, c_{N - 1}, c_M )$ satisfies
Theorem~\ref{Tthreshold}(2). Similarly, Theorem~\ref{Tthreshold} shows
the surprising result that to preserve positivity on all of
$\bp_N( \overline{D}( 0, \rho ) )$ is equivalent to preserving
positivity on the much smaller subset of real rank-one matrices
$\bp_N^1( ( 0, \rho ) )$.
\end{remark}

\subsection{Positivity preservers: sufficient conditions}\label{Sanalytic}

At this point, extensions of Theorem~\ref{Tthreshold} to more general
classes of functions are within reach. We first introduce some
notation.

\begin{definition}\label{Drayleigh}
Given $K \subset \C$, functions $g$, $h: K \to \C$, and a set of
positive semidefinite matrices
$\bp \subset \bigcup_{N = 1}^\infty \bp_N( K )$, let
$\mathcal{C}( h; g; \bp )$ be the smallest real number such that
\begin{equation}\label{Erayleigh}
g[ A ] \leq \mathcal{C}( h; g; \bp ) \cdot h[ A ], %
\qquad \forall A \in \bp.
\end{equation}

\noindent In other words, the constant $\mathcal{C}( h; g; \bp )$ is the
\emph{extreme critical value} of the family of linear pencils
$\{ g[A] - \R h[A] : A \in \bp \}$. If
$h_\bc( z ) := c_0 + c_1 z + \cdots + c_{n - 1} z^{n - 1}$ is a polynomial
with coefficients $\bc := ( c_0, \ldots, c_{n - 1} )$, then we let
\[
\mathcal{C}( \bc; g; \bp ) := \mathcal{C}( h_\bc; g; \bp )
\]

\noindent to simplify the notation. Similarly, if
$\bp = \bp_N( \overline{D}( 0, \rho ) )$ for some
$\rho \in ( 0, \infty )$ and an integer $N \geq 1$, then we let
\[ 
\mathcal{C}( h; g; N, \rho ) :=
\mathcal{C}( h; g; \bp_N( \overline{D}( 0, \rho ) ) ).
\]
Finally, let
\[
\mathcal{C}( \bc; g; N, \rho ) := %
\mathcal{C}( c_0 + \cdots + c_{n - 1} z^{n - 1}; g; %
\bp_N( \overline{D}( 0, \rho ) ) )
\]
for $\rho \in ( 0, \infty )$ and integers $n$, $N \geq 1$.
\end{definition}

Note that the notation introduced in Definition~\ref{Drayleigh} is
consistent with the notation used in the previous sections of the paper.
Also, note for future use that the constant $\mathcal{C}( h; g; \bp )$
is non-increasing in the last argument, as well as $\R_+$-subadditive in
the second argument:
\begin{equation}\label{Esubadditive}
\mathcal{C}( h; \lambda_1 g_1 + \cdots + \lambda_n g_n; \bp ) %
\leq \sum_{j = 1}^n \lambda_j \mathcal{C}( h; g_j ; \bp ), \qquad %
\forall \lambda_1, \ldots, \lambda_n \geq 0.
\end{equation}

We now show how Theorem~\ref{Tthreshold} can naturally be extended to
arbitrary polynomials.

\begin{corollary}
Fix a bounded subset $K \subset \C$, and integers $N \geq 1$ and
$M \geq 0$. There exists a universal constant
$\mathfrak{h}_{N, M}( K ) > 0$ depending only on $N$, $M$, and $K$,
with the following property: for any integer $R \geq 0$ and any
polynomial $f( x ) = x^R \sum_{k = 0}^{N + M} c_k x^k$ with real
coefficients such that
\begin{enumerate}
\item $c_0$, \ldots, $c_{N - 1} > 0$, and 
\item
$\min\{ c_k : 0 \leq k \leq N-1 \} \geq \mathfrak{h}_{N, M}( K ) %
\cdot \max\{ |c_l| : c_l < 0 \}$,
\end{enumerate}
we have $f[ - ] : \bp_N( K ) \to \bp_N( \C )$.
\end{corollary}

\begin{proof}
Without loss of generality, we may assume $R = 0$; the general case
follows from the Schur product theorem.  Now fix $\rho > 0$ such that
$K \subset \overline{D}( 0, \rho )$, and let
\begin{equation}\label{EgeneralCte}
\mathfrak{h}_{N, M}( K ) := \sum_{m = 0}^M %
\mathcal{C}( ( 1, \ldots, 1 ); z^{N+m}; N, \rho )
\in ( 0, \infty ),
\end{equation}
where the first argument has $N$ ones. Now given $A \in \bp_N( K )$,
one finds
\[
f[ A ] = \sum_{j = 0}^{N - 1} c_j A^{\circ j} + %
\sum_{j = N}^{N + M} c_j A^{\circ j} \geq %
\min_{0 \leq k \leq N - 1} c_k  \sum_{j = 0}^{N - 1} A^{\circ j} + %
\min_{l \geq N : c_l < 0} c_l \sum_{j = N}^{N + M} A^{\circ j}.
\]
The result now follows immediately from Equation~\eqref{Esubadditive}.
\end{proof}

Next we show how our main result extends from polynomial to analytic
functions.

\begin{proof}[Proof of Theorem~\ref{Tanalytic}]
Without loss of generality, we may assume that
$c_M \geq 0$ for all $M \geq N$, and, via a standard approximation
argument, that the series in the statement is analytic in the disc
$D( 0, \rho + \epsilon )$, where $\epsilon  > 0$.

The first part is immediate from Theorem~\ref{Tthreshold}
and Equation~\eqref{Esubadditive}. To establish the
bound~\eqref{Eanalytic-bound}, note first that, for
$0 \leq j \leq N - 1$, we have
\begin{align*}
\binom{M}{j} \binom{M - j - 1}{N - j - 1} & = %
\frac{M!}{j! (M-N)! (N - j - 1)! ( M - j )} \\
& \leq \frac{M!}{j! ( M - N + 1 )! ( N - j - 1 )!}
 = \binom{N - 1}{j} \binom{M}{N - 1}.
\end{align*}
Using the above analysis and Tonelli's theorem, we compute:
\begin{align*}
&\sum_{M = N}^\infty c_M \mathcal{C}( \bc; z^M; N, \rho ) \\
 & = \sum_{j = 0}^{N - 1} \frac{\rho^{N - j}}{c_j} %
\sum_{M = N}^\infty c_M \binom{M}{j}^2 %
\binom{M - j - 1}{N - j - 1}^2 \rho^{M - N} \\
 & \leq \! \frac{1}{( N - 1 )!^2} \sum_{j = 0}^{N - 1} %
\binom{N - 1}{j}^2 \frac{\rho^{N - j}}{c_j} \!\! %
\sum_{M = N}^\infty c_M \rho^{M - N} %
\prod_{k = 2}^N ( M - N + k )^2 \\
 & \leq \! \frac{1}{( N - 1 )!^2} \sum_{j = 0}^{N - 1} %
\binom{N - 1}{j}^2 \frac{\rho^{N - j}}{c_j} \!\! %
\sum_{M = N}^\infty c_M \rho^{M - N} %
\prod_{k = 2}^N \frac{( 2 ( M - N ) + k + 1 )( 2 M - N + k )}{2} \\
 & = \frac{2^{-( N - 1 )}}{( N - 1 )!^2} \sum_{j = 0}^{N - 1} %
\binom{N - 1}{j}^2 \frac{\rho^{N - j - 1}}{c_j} %
\sum_{M = N}^\infty c_M \prod_{k=0}^{2N-3} ( 2 M - k ) \cdot %
( \sqrt{\rho} )^{2 M - 2 N + 2}.
\end{align*}
Notice that $g_2$ is analytic on $D( 0, \sqrt{\rho + \epsilon} )$,
since $g$ is analytic on $D( 0, \rho + \epsilon )$. Therefore the
inner sum is precisely the $( 2 N - 2 )$th derivative of
$g_2( z ) = g_+( z^2 )$, evaluated at $z = \sqrt{\rho}$. This
concludes the proof.
\end{proof}

\begin{remark}
The quantity $g_2^{(2 N - 2 )}( \sqrt{\rho} )$ can be written in terms
of the derivatives of $g_+$ at $z = \rho$; it may be shown by
induction that
\begin{equation}
\frac{\rd^n}{\rd x^n} \Bigl( g_+( x^2 ) \Bigr) = %
\sum_{k = 0}^{\lfloor n / 2 \rfloor} \frac{n!}{( n - 2 k)! k!} %
( 2 x )^{n - 2 k} g_+^{( n - k )}( x^2 ) \qquad ( n \geq 0 ).
\end{equation}
This shows that one has explicit bounds for
$\mathcal{C}( \bc; g; N, \rho ) \leq \mathcal{C}( \bc; g_+; N, \rho )$
in terms of the derivatives of $g_+$ at $\rho$.
\end{remark}

We conclude this part by showing how Theorem~\ref{Tthreshold} yields
an asymptotically sharp bound for the matrix cube problem for Hadamard
powers.

\begin{proof}[Proof of Corollary~\ref{Ccube}]
Note that $\mathcal{U}[ \eta ] \subset \bp_N( \C )$ if and only if
\[
A_0 - \eta \sum_{m = 1}^{M + 1} A_m = %
c_0 \one{N} + A + \cdots + c_{N - 1} A^{\circ ( N - 1 )} - %
\eta ( A^{\circ N} + \cdots + A^{\circ ( N + M )} ) \in \bp_N( \C ).
\]
Thus the first implication follows by Theorem~\ref{Tanalytic}. The
other implication follows by setting $u_1 = \cdots = u_M = 0$ in
Equation~\eqref{Ecube} and applying Theorem~\ref{Tthreshold}.

It remains to show the asymptotics in \eqref{Easymptotic}. For this it
suffices to show that
\begin{equation}\label{Estirling}
\lim_{N \to \infty} %
\frac{\mathcal{C}( \bc; z^{N+m}; N, \rho )}%
{\mathcal{C}( \bc; z^{N+M}; N, \rho )} = 0 \qquad %
( 0 \leq m \leq M - 1 ).
\end{equation}
In turn, to show \eqref{Estirling} we first fix $N$ and $\rho$, and
write out each summand in the numerator and denominator of
\eqref{Estirling} as follows:
\[
a( m, j ) := \binom{N + m}{j}^2 \binom{N + m - j - 1}{N - j - 1}^2 %
\frac{\rho^{N + m - j}}{c_j}.
\]
Next we bound the ratios of the summands for fixed $N$, $m$, $M$, and
$\rho$, and for $0 \leq j \leq N - 1$:
\begin{align*}
\frac{a( m, j )}{a( M, j )} & = \Bigl( %
\frac{( N + m )! M! ( N + M - j )}{m! ( N + M )! ( N + m - j )}
\Bigr)^2 \rho^{m - M} \\
 & = \frac{1}{( N + M )^2 \cdots ( N + m + 1 )^2} %
\Bigl( \frac{M!}{m!} \cdot \frac{N + M - j}{N + m - j} \Bigr)^2 %
\rho^{m - M} \\
 & \leq \frac{1}{N^{2 ( M - m )}} %
\Bigl( \frac{M!}{m!} \cdot \frac{N + M - j}{N + m - j} \Bigr)^2 %
\rho^{m - M} \\
 & \leq \frac{1}{N^{2 ( M - m )}} %
\Bigl( \frac{( M + 1 )!}{( m + 1 )!} \Bigr)^2 \rho^{m - M},
\end{align*}
where the final step follows from the observation that if
$M > m \geq 0$ are fixed, and $N - j = \alpha \geq 1$, then
$\frac{\alpha + M}{\alpha + m} = 1 + \frac{M - m}{\alpha + m}$ has
its global maximum at $\alpha = 1$. Now let
\[
b( m, M, \rho ) := %
\Bigl( \frac{( M + 1 )!}{( m + 1 )!} \Bigr)^2 \rho^{m - M},
\]
so that $a( m, j ) \leq a( M, j ) b( m, M, \rho ) N^{-2 ( M - m )}$
for $0 \leq j \leq N - 1$. Hence
\[
0 \leq
\frac{\mathcal{C}( \bc; z^{N+m}; N, \rho )}%
{\mathcal{C}( \bc; z^{N+M}; N, \rho )} \leq %
b( m, M, \rho ) N^{-2 ( M - m )} \to 0
\]
as $N \to \infty$, as required.
\end{proof}

\subsection{Case study: $2 \times 2$ matrices}

Remark that Theorem~\ref{Tthreshold} holds for any integer $N \geq 1$.
When $N = 2$, it is possible to prove a characterization result for
polynomials preserving positivity on $\bp_2$, for a more general
family of polynomials than in Theorem~\ref{Tthreshold}. Along these
lines, we conclude the present section with the following result.

\begin{theorem}\label{T2x2}
Given non-negative integers $m < n < p$, let
$f( x ) = c_m x^m + c_n x^n + c_p x^p$, with $c_m$ and $c_n$ both
non-zero. The following are equivalent.
\begin{enumerate}
\item The entrywise function $f[ - ]$ preserves positivity on
$\bp_2( [ 0, 1 ] )$.

\item The entrywise function $f[ - ]$ preserves positivity on
$\bp_2^1( [ 0, 1 ])$.

\item $c_m$, $c_n > 0$, and
\[
c_p \geq %
\frac{-c_m c_n ( n - m )^2}{c_m ( p - m )^2 + c_n ( p - n )^2}.
\]
\end{enumerate}
Note that if $c_p < 0$ then such an entrywise function does not
preserve positivity on $\bp_3( [ 0, 1 ] )$.
\end{theorem}

In the special case $m = 0$ and $n = 1$, note that the bound
$\frac{- c_m c_n ( n - m )^2}{c_m ( p - m )^2 + c_n ( p - n )^2}$
reduces to the constant $-\mathcal{C}( \bc; p; 2, 1 )^{-1}$.

\begin{proof}
Clearly (1) implies (2). To see why (2) implies (3), note that $c_m$,
$c_n > 0$ by Lemma~\ref{Lpos-coeff}. Now let $u_1$, $u_2 \in [ 0, 1 ]$
be distinct and let
$f_m( u_1, u_2 ) := ( u_1^m - u_2^m ) / ( u_1 - u_2 )$ denote the
divided difference of their $m$th powers. With $\bu = ( u_1, u_2 )^T$,
\begin{align*}
0 \leq \det f[ \bu \bu^T ] & = %
f( u_1^2 ) f( u_2^2 ) - f( u_1 u_2 )^2 \\
 & = c_m c_n ( u_1 u_2 )^{2 m} ( u_2^{n - m} - u_1^{n - m} )^2 + %
c_m c_p ( u_1 u_2 )^{2 m} ( u_2^{p - m} - u_1^{p - m} )^2 \\
 & \qquad + c_n c_p ( u_1 u_2 )^{2 n} ( u_2^{p - n} - u_1^{p - n} )^2,
\end{align*}
which implies that
\[
c_p \geq \frac{-c_m c_n f_{n - m}( u_1, u_2)^2}%
{c_m f_{p - m}( u_1, u_2 )^2 + c_n ( u_1 u_2 )^{2 ( n - m )} %
f_{p - n}( u_1, u_2 )^2}.
\]
Letting $u_1 \to 1 = u_2$ yields (3).

Finally, suppose (3) holds. Then (1) holds if and only
if $f( x y )^2 \leq f( x^2 ) f( y^2 )$ for all $x$, $y \in [ 0, 1 ]$
and $f$ is non-decreasing and non-negative on $[ 0, 1 ]$; see, for
example, \cite[Theorem~2.5]{GKR-lowrank}. By
\cite[Exercise~2.4.4]{niculescu}, the first of these conditions is
satisfied if and only if the function
\[
\Psi_f( x ) := %
x ( f''( x ) f( x ) - f'( x )^2 ) + f( x ) f'( x ) \geq 0
\]
for all $x \in ( 0, 1 )$. A short calculation gives that
\[
\Psi_f( x ) = x^{m + n - 1} \bigl( c_m c_n ( n - m )^2 + %
c_m c_p ( p - m )^2 x^{p - n} + c_n c_p ( p - n )^2 x^{p - m} \bigr),
\]
and so
\[
\Psi_f( x ) \geq 0 \iff c_p \geq %
\frac{-c_m c_n ( n - m )^2}%
{c_m ( p - m )^2 x^{p - m} + c_n ( p - n )^2 x^{p - n}}.
\]
The final term has its infimum on $( 0 , 1 )$ when $x = 1$, so we
obtain the first condition. Next, note that $f$ is non-decreasing on
$[ 0, 1 ]$ if and only if
\begin{align*}
f'( x ) & = m c_m x^{m - 1} + n c_n x^{n - 1} + p c_p x^{p - 1} \geq 0
\quad \forall x \in ( 0, 1 ) \\
 \iff c_p & \geq \frac{-m c_m  x^{m - p} - n c_n  x^{n - p}}{p} %
\quad \forall x \in ( 0, 1 ) \\
\iff c_p & \geq \frac{-m c_m - n c_n}{p}.
\end{align*}
Thus, it suffices to show that
\[
\frac{m c_m + n c_n}{p} \geq %
\frac{c_m c_n ( n - m )^2}{c_m ( p - m )^2 + c_n ( p - n )^2},
\]
but this holds because
\[
c_m c_n ( m ( p - n )^2 + n ( p - m )^2 - p ( n - m )^2 ) = %
c_m c_n ( m + n ) ( p - m ) ( p - n ) \geq 0.
\]
Thus (3) implies that $f$ is non-decreasing on $[0,1]$. In turn, this
implies that $f$ is non-negative on $[ 0, 1 ]$, since
$f( x ) \geq f( 0 ) \geq 0$ for $x \in [ 0, 1 ]$. Hence (3) implies
(1).

The final assertion is an immediate consequence of
Lemma~\ref{Lpos-coeff}.
\end{proof}

\section{Rayleigh quotients}\label{Srayleigh}

Recall from Theorem~\ref{Crayleigh} that Theorem~\ref{Tthreshold} can
be reformulated as an extremal problem, involving the boundedness of
the generalized Rayleigh quotient for
$\sum_{j = 0}^{N - 1} c_j A^{\circ j}$ and $A^{\circ M}$, taken over
all matrices $A \in \bp_N( \overline{D}( 0, \rho ) )$. We now consider
an alternate approach to proving Theorem~\ref{Tthreshold}, by first
considering the analogue of Theorem~\ref{Crayleigh} for a single
matrix. In this case a sharp bound may be obtained as follows.

\begin{proposition}\label{Palgebra}
Fix integers $N \geq 1$ and $M \geq 0$, positive scalars $c_0$,
\ldots, $c_{N - 1} > 0$, and a non-zero matrix $A \in \bp_N( \C )$.
Then, with the notation as in Theorem~\ref{Crayleigh} and
Definition~\ref{Drayleigh},
$\mathcal{K}( A ) \subset \ker A^{\circ M}$, and the corresponding
extreme critical value $\mathcal{C}( \bc; z^M; A )$ is finite:
\[
\mathcal{C}( \bc; z^M; A )^{-1} = %
\min_{\bv \in S^{2 N - 1} \cap \mathcal{K}(A)^\perp} %
\frac{\bv^* \Bigl( \sum_{j = 0}^{N - 1} c_j A^{\circ j} \Bigr) \bv}%
{\bv^* A^{\circ M} \bv},
\]
where $S^{2 N - 1}$ is the unit sphere in $\C^N$.
\end{proposition}
In particular,
\[
\bv^* \biggl( \sum_{j = 0}^{N - 1} c_j A^{\circ j} \biggr) \bv \geq %
\mathcal{C}( \bc; z^M; A )^{-1} \cdot \bv^* A^{\circ M} \bv
\]
for all $\bv \in \C^N$.

Note, moreover, that the minimum in Proposition~\ref{Palgebra} is
attained for every non-zero matrix $A \in \bp_N( \C )$ and integer
$M \geq 0$, and over a compact set that is independent of~$M$.

The proof of Proposition~\ref{Palgebra}, as well as a closed-form
expression for the constant $\mathcal{C}( \bc; z^M; A )$, is
immediate from the following two results.

\begin{proposition}\label{Ppos-constant}
Given $N \geq 1$ and $C$, $D \in \bp_N( \C )$, the following are
equivalent.
\begin{enumerate}
\item If $\bv^* C \bv = 0$ for some $\bv \in \C^N$, then
$\bv^* D \bv = 0$.

\item $\ker C \subset \ker D$.

\item There exists a smallest positive constant $\mathfrak{h}_{C, D}$
such that $\bv^* C \bv \geq \mathfrak{h}_{C, D} \cdot \bv^* D \bv$ for
all $\bv \in \C^N$.
\end{enumerate}

\noindent Moreover, if (1)--(3) hold and $D \neq 0$, then the constant
is computable as an extremal generalized Rayleigh quotient, as
follows:
\begin{equation}\label{Egrquot}
\mathfrak{h}_{C, D}^{-1} = %
\sup_{\bv \not\in \ker D} \frac{\bv^* D \bv}{\bv^* C \bv} = %
\varrho( C^{\dagger / 2} D C^{\dagger / 2} ) = \varrho( X^* C^\dagger X ),
\end{equation}
where $C^{\dagger / 2}$ and $C^\dagger \in \bp_N( \C )$ denote
respectively the square root of the Moore--Penrose inverse
and the Moore--Penrose inverse of $C \in \bp_N( \C )$,
$\varrho( - )$ denotes the spectral radius, and $X$ is any matrix such
that $D = X X^*$.
\end{proposition}

\begin{proof}
Clearly (3) implies (1). That (1) is equivalent to (2) is also
immediate, given the following reasoning:
\begin{alignat*}{2}
\bv^* C \bv = 0 \quad & \iff \quad
\bv^* C^{1 / 2} \cdot C^{1 / 2} \bv = 0 \quad && \\
 & \iff \quad C^{1 / 2} \bv = 0 \quad && \implies \quad C \bv = 0 %
\implies \quad \bv^* C \bv = 0.
\end{alignat*}
We now show that (1) implies (3). Given a matrix $C$, denote its
kernel and the orthogonal complement of its kernel by $K_C$ and
$K_C^\perp$, respectively. If $\bv \in K_D$ then any choice of
constant $\mathfrak{h}_{C, D}$ would suffice in (3), so we may
restrict ourselves to obtaining such a constant when
$\bv \not\in K_D$. Write $\bv = \bv_C + \bv_C^\perp$, with
$\bv_C \in K_C$ and $\bv_C^\perp \in K_C^\perp$, and note that
$\bv_C^\perp \neq 0$. Now compute:
\[
\frac{\bv^* C \bv}{\bv^* D \bv} = %
\frac{( \bv_C^\perp )^* C \bv_C^\perp}%
{( \bv_C^\perp )^* D \bv_C^\perp} = %
\frac{( \bv_C^\perp / \| \bv_C^\perp \| )^* C %
( \bv_C^\perp / \| \bv_C^\perp \| )}%
{( \bv_C^\perp / \| \bv_C^\perp \| )^* D %
( \bv_C^\perp / \| \bv_C^\perp \| )}.
\]
It follows by setting $\bw := \bv_C^\perp / \| \bv_C^\perp \|$ that
\begin{equation}\label{Epos-constant}
\mathfrak{h}_{C, D} := %
\inf_{\bv \not\in K_D} \frac{\bv^* C \bv}{\bv^* D \bv} = %
\min_{\bw \in S^{2 N - 1} \cap K_C^\perp} %
\frac{\bw^* C \bw}{\bw^* D \bw},
\end{equation}
where $S^{2 N - 1}$ denotes the unit sphere in $\C^N$. The right-hand
side is the minimizer of a continuous and positive function over a
compact set, hence equals a positive real number. Thus
$\mathfrak{h}_{C, D} > 0$, as desired.

We now establish Equation~\eqref{Egrquot}, by following the theory of
the Kronecker normal form for a matrix pencil, as developed
in~\cite[Chapter~X, \S6]{Gantmacher_Vol2}. Because both sides of the
last equality in Equation~\eqref{Egrquot} remain unchanged under a
unitary change of basis, we may assume that $C$ is diagonal, say
$C = \diag( \lambda_1, \ldots, \lambda_r, 0, \ldots, 0)$ with
$\lambda_1 \geq \cdots \geq \lambda_r > 0$. Pre-multiplying by
$C^{1 / 2} = %
\diag( \lambda_1^{1 / 2}, \ldots, \lambda_r^{1 / 2}, 0, \ldots, 0)$
preserves $K_C^\perp \setminus \{ 0 \}$, and therefore, setting
$\bw = C^{1 / 2} \bv_C^\perp$,
\begin{align*}
\mathfrak{h}_{C, D}^{-1} & = %
\sup_{\bv \not\in \ker D} \frac{\bv^* D \bv}{\bv^* C \bv} = %
\sup_{\bv_C^\perp \in K_C^\perp \setminus \{ 0 \}} %
\frac{( \bv_C^\perp )^* D \bv_C^\perp}%
{( \bv_C^\perp )^* C \bv_C^\perp} = %
\sup_{\bw \in K_C^\perp \setminus \{ 0 \}} %
\frac{\bw^* C^{\dagger/2} D C^{\dagger/2} \bw}{\bw^* \bw} \\
 & = \sup_{\bw \in S^{2N-1} \cap K_C^\perp} %
\bw^* C^{\dagger / 2} D C^{\dagger / 2} \bw \\
 & \leq \sup_{\bw \in S^{2 N - 1}} %
\bw^* C^{\dagger / 2} D C^{\dagger / 2} \bw = %
\varrho( C^{\dagger / 2} D C^{\dagger / 2} ).
\end{align*}
Moreover, any non-zero eigenvector of
$C^{\dagger / 2} D C^{\dagger / 2}$ with non-zero eigenvalue has its
last $n - r$ coordinates all equal to zero, since
$C^{\dagger / 2} = %
\diag( \lambda_1^{-1 / 2}, \ldots, \lambda_r^{-1 / 2}, 0, \ldots, 0)$.
It follows that the unit-length eigenvectors corresponding to the
eigenvalue $\varrho( C^{\dagger / 2} D C^{\dagger / 2})$ lie in
$K_C^\perp$, thereby proving the result. 
The last equality in Equation~\eqref{Egrquot} is an immediate consequence
of the tracial property of the spectral radius.
\end{proof}

\begin{lemma}\label{Lalgebra}
Fix integers $m \geq N \geq 1$ and $M \geq 0$, and matrices $A$,
$B \in \bp_N( \C )$. The matrices
$h( A, B ) := B \circ \sum_{j = 0}^{m - 1} A^{\circ j}$ and
$g( A, B ) := B \circ A^{\circ M}$ are such that
$\ker h( A, B ) \subset \ker g( A, B )$.
\end{lemma}

\begin{proof}
Suppose $\bv \in \ker h( A, B )$, so that $\bv^* h( A, B ) \bv = 0$.
Since $B \circ A^{\circ j} \in \bp_N( \C )$, it follows that
$\bv^* ( B \circ A^{\circ j} ) \bv = 0$ and
$( B \circ A^{\circ j} ) \bv = 0$ for $0 \leq j \leq m - 1$; in
particular, the lemma is true if $M < m$.
If, instead, $M \geq N$, then apply Lemma~\ref{Lpower_exp}, together
with the fact that $D ( C \circ B ) = ( D C ) \circ B$ if $B$, $C$
and $D$ are square matrices, with $D$ diagonal, to see that
\begin{align*}
h( A, B ) \bv = 0 \quad & \implies \quad %
( A^{\circ j} \circ B ) \bv = 0 \quad ( 0 \leq j \leq N - 1 ) \\
 & \implies \quad %
\sum_{j = 0}^{N - 1} D_{M, j}( A ) ( A^{\circ j} \circ B ) \bv = 0 %
\quad \implies \quad ( A^{\circ M} \circ B ) \bv = 0.
\qedhere
\end{align*}
\end{proof}

\begin{remark}
Note that Theorem~\ref{Crayleigh} does not hold in much greater
generality, i.e.,~for more general functions $g$, $h$ than the
polynomials used in Lemma~\ref{Lalgebra}. For instance,
fix $\delta > 0$ and consider continuous functions $g$,
$h : [ 0, \delta ) \to [ 0, \infty )$ such that $g( 0 ) = h( 0 ) = 0$
and $g( x )$, $h( x ) > 0$ on $( 0, \delta )$. Given $N \geq 2$ and
$t \in \R$, let the $N \times N$ matrix
$A( t ) := \diag( t, 1, \ldots, 1 )$. There need not exist a universal
constant $\mathfrak{h} > 0$ such that
\begin{equation}\label{eqn:noconst}
\bv^T h[ A( t ) ] \bv \geq \mathfrak{h} \cdot \bv^T g[ A( t ) ] \bv %
\qquad \forall t \in ( 0, \delta ), \ \bv \in \R^N.
\end{equation}
To see this, note that $A( t )$, $h[ A( t ) ]$ and $g[ A( t ) ]$ are
positive definite for $t \in ( 0, \delta )$ and singular at
$t = 0$. Thus, by standard facts about the generalized Rayleigh
quotient, if \eqref{eqn:noconst} holds, then
\begin{align*}
\mathfrak{h} \leq \inf_{t \in ( 0, \delta )} %
\inf_{\bv \in \R^N \setminus \{ 0 \}} %
\frac{\bv^T h[ A( t ) ] \bv}{\bv^T g[ A( t ) ] \bv} & = %
\inf_{t \in ( 0, \delta )} %
\lambda_{\min}( g[ A( t )]^{-1 / 2} h[ A( t ) ] g[ A( t ) ]^{-1 / 2} )
\\
 & = \inf_{t \in ( 0, \delta )} \min\{ 1, h( t ) / g( t ) \},
\end{align*}
where $\lambda_{\min}( C )$ denotes the smallest eigenvalue of the
positive semidefinite matrix $C$. If $h( t ) = t^a g( t )$ for some
$a > 0$, then this infimum, and so $\mathfrak{h}$, equals zero. Note,
however, that such a case does not occur in Theorem~\ref{Crayleigh}.
\end{remark}

We now observe that Proposition~\ref{Palgebra} and the explicit
formula for the extreme critical value of the matrix pencil, as given
in Equation~\eqref{Egrquot}, allow us to provide a closed-form
expression for the constant $\mathcal{C}( \bc; z^M; A )$ for rank-one
matrices.

\begin{corollary}\label{Cgrquot}
Given integers $N \geq 1$ and $M \geq 0$, positive scalars $c_0$,
\ldots, $c_{N - 1} > 0$, and a matrix
$A = \bu \bu^* \in \bp_N^1(\C) \setminus \{ {\bf 0}_{N \times N} \}$, the
identity
\[
\mathcal{C}( \bc; z^M; A ) = %
( \bu^{\circ M} )^* \biggl( %
\sum_{j = 0}^{N - 1} c_j \bu^{\circ j} ( \bu^{\circ j} )^* %
\biggr)^\dagger \bu^{\circ M}
\]
holds.
In particular, given $\rho > 0$,
\[
\mathcal{C}( \bc; z^M; \rho {\bf 1}_{N \times N} ) = %
\rho^M \biggl( \sum_{j = 0}^{N - 1} c_j \rho^j \biggr)^{-1} \leq %
\mathcal{C}( \bc; z^M; N, \rho ),
\]
with equality if and only if $N = 1$.
\end{corollary}

\begin{proof}
Applying Proposition~\ref{Palgebra} and Equation~\eqref{Egrquot}
with $C = h_\bc[ A ] := \sum_{j = 0}^{N - 1} c_j A^{\circ j}$ and
$D = A^{\circ M} = \bu^{\circ M} ( \bu^{\circ M} )^*$, we obtain
\[
\mathcal{C}( \bc; z^M; A ) = \varrho( \bv \bv^* ), \quad
\text{where} \quad \bv = h_\bc[ \bu \bu^* ]^{\dagger / 2} \bu^{\circ M}.
\]
Now, it is well known that the unique non-zero eigenvalue of
$\bv \bv^*$ equals
\[
\bv^* \bv = (\bu^{\circ M})^* h_\bc[ A ]^\dagger \bu^{\circ M} = %
( \bu^{\circ M} )^* \biggl( %
\sum_{j = 0}^{N - 1} c_j \bu^{\circ j} ( \bu^{\circ j} )^* %
\biggr)^\dagger \bu^{\circ M},
\]
which shows the first assertion. For the next, set
$\bu = \sqrt{\rho} ( 1, \ldots, 1 )^*$, and use
the fact that
$( \alpha \bu \bu^* )^\dagger = %
\alpha^{-1} ( \bu^* \bu )^{-2} (\bu \bu^* )$ for non-zero $\alpha$
and $\bu$, to obtain that
\[
\mathcal{C}( \bc; z^M; \rho {\bf 1}_{N \times N} ) = %
\rho^M \biggl( \sum_{j = 0}^{N - 1} c_j \rho^j \biggr)^{-1}.
\]

It remains to show the last inequality. When $N = 1$,
\[
\mathcal{C}( \bc; z^M; N, \rho ) \sum_{j = 0}^{N - 1} c_j \rho^j = %
\frac{\rho^M}{c_0} \cdot c_0 = \rho^M.
\]
If instead $N > 1$, then
\[
\mathcal{C}( \bc; z^M; N, \rho ) \sum_{j = 0}^{N - 1} c_j \rho^j > %
\sum_{j = 0}^{N - 1} \binom{M}{j}^2 \binom{M - j - 1}{N - j - 1}^2 %
\rho^M > \rho^M.
\qedhere
\]
\end{proof}

\begin{remark}
If $A = \bu \bu^*$, with $\bu$ having pairwise-distinct entries, then
$h_\bc[ A ]$ is the sum of $N$ rank-one matrices with linearly independent
column spaces, and hence is non-singular. In particular, it is
possible to write $\mathcal{C}( \bc; z^M; A )$ for such $A$ in an
alternate fashion: if $h_\bc( z ) = \sum_{j = 0}^{N - 1} c_j z^j$, where
$c_0$, \ldots, $c_{N - 1} > 0$ then
\[
\mathcal{C}( \bc; z^M; A ) = %
( \bu^{\circ M} )^* h_\bc[ \bu \bu^* ]^{-1} \bu^{\circ M} = %
1 - \det h_\bc[ \bu \bu^* ]^{-1} %
\det \begin{pmatrix} h_\bc[ \bu \bu^*] & \bu^{\circ M} \\
 (\bu^{\circ M})^* & 1 \end{pmatrix}.
\]
\end{remark}

For the optimization-oriented reader, we now restate the main result
as an extremal problem that follows immediately from
Theorem~\ref{Crayleigh} and Proposition~\ref{Palgebra}:
\begin{align}\label{Erephrase}
& \inf_{A \in \bp_N^1( ( 0, \rho ) )} %
\min_{\bv \in S^{2 N - 1} \cap \mathcal{K}( A )^\perp} %
\frac{\bv^* \Bigl( \sum_{j = 0}^{N - 1} c_j A^{\circ j} \Bigr) \bv}%
{\bv^* A^{\circ M} \bv} \notag \\
 & = \inf_{A \in \bp_N( \overline{D}( 0, \rho ) )} %
\min_{\bv \in S^{2 N - 1} \cap \mathcal{K}( A )^\perp} %
\frac{\bv^* \Bigl( \sum_{j = 0}^{N - 1} c_j A^{\circ j} \Bigr) \bv}%
{\bv^* A^{\circ M} \bv} \notag \\
 & = \biggl( \sum_{j = 0}^{N - 1} \binom{M}{j}^2 %
\binom{M - j - 1}{N - j - 1}^2 \frac{\rho^{M - j}}{c_j} \biggr)^{-1},
\end{align}
or, equivalently,
\begin{equation}\label{Erephrase2}
\sup_{A \in \bp_N^1( ( 0, \rho ) )} \mathcal{C}( \bc; z^M; A ) = %
\sup_{A \in \bp_N( \overline{D}( 0, \rho ) )} %
\mathcal{C}( \bc; z^M; A ) = \mathcal{C}( \bc; z^M; N, \rho ),
\end{equation}
where
$\mathcal{C}( \bc; z^M; A ) = %
\varrho( h_\bc[ A ]^{\dagger / 2} A^{\circ M} h_\bc[ A ]^{\dagger / 2} )$,
with $h_\bc( z ) = \sum_{j = 0}^{N - 1} c_j z^j$.

\begin{remark}\label{discts}
Note from the previous line that
$A \mapsto \mathcal{C}( \bc; z^M; A )$ is continuous on the subset of
the cone $\bp_N( \C )$ where $\det h_\bc[ A ] \neq 0$. The obstacle to
establishing~\eqref{Erephrase} and~\eqref{Erephrase2}, and so
proving Theorems~\ref{Tthreshold} and~\ref{Crayleigh}, resides in
the fact that this function is not continuous on the whole of
$\bp_N( \C )$ for $N > 1$. In particular, it is not continuous at the
matrix $A = \rho \one{N} \in \bp_N^1( \overline{D}( 0, \rho ) )$, as
shown in the calculations for Equation~\eqref{Elowerbound} and
Corollary~\ref{Cgrquot} above. However, these calculations also reveal
that the sharp constant $\mathcal{C}( \bc; z^M; N, \rho )$ is obtained
by taking the supremum of the Rayleigh quotient over the one-parameter
family of rank-one matrices
\[
\{ \rho \bu( t ) \bu( t )^T : \bu( t ) := %
( 1 - t, \ldots, 1 - N t )^T, \ t \in ( 0, 1 / N ) \}.
\]
\end{remark}

\section{The simultaneous kernels}\label{Skernels}

Prompted by the variational approach of the previous section, a
description of the kernel
$\mathcal{K}( A ) = %
\ker( c_0 \one{N} + c_1 A + \cdots + c_{N - 1} A^{\circ ( N - 1 )} )$
for any $A \in \bp_N( \C )$ is in order. As we prove below, this
kernel does not depend on the choice of scalars $c_j > 0$, and
coincides with the simultaneous kernel
\[
\bigcap_{n \geq 0} \ker A^{\circ n}
\] 
of the Hadamard powers of $A$. A refined structure of the matrix $A$,
based on an analysis of the kernels of iterated Hadamard powers, is
isolated in the first part of this section.

\subsection{Stratifications of the cone of positive semidefinite
matrices}

We introduce a novel family of stratifications of the cone
$\bp_N( \C )$, that are induced by partitions of the set
$\{ 1, \ldots, N \}$. Each stratification is subjacent to a block
decomposition of a positive semidefinite $N \times N$ matrix, with
diagonal blocks having rank one.  In addition, the blocks exhibit a
remarkable homogeneity with respect to subgroups of the group
$\C^\times$, the multiplicative group of non-zero complex numbers.

\begin{theorem}\label{Tgroup}
Fix a subgroup $G \subset \C^\times$, an integer $N \geq 1$,
and a non-zero matrix $A \in \bp_N( \C )$.
\begin{enumerate}
\item Suppose $\{ I_1, \ldots, I_k \}$ is a partition of
$\{ 1, \ldots, N \}$ satisfying the following two conditions.
\begin{enumerate}
\item Each diagonal block $A_{I_j}$ of $A$ is a submatrix having rank
at most one, and $A_{I_j} = \bu_j \bu_j^*$ for a unique
$\bu_j \in \C^{| I_j |}$ with first entry
$\bu_{j, 1} \in [ 0, \infty )$.

\item The entries of each diagonal block $A_{I_j}$ lie in a single
$G$-orbit.
\end{enumerate}
Then there exists a unique matrix $C = ( c_{i j})_{i,j = 1}^k$ such
that $c_{i j} = 0$ unless $\bu_i \neq 0$ and $\bu_j \neq 0$, and $A$
is a block matrix with
\[
A_{I_i \times I_j} = c_{i j} \bu_i \bu_j^* \qquad %
( 1 \leq i, j \leq k ).
\]
Moreover, the entries of each off-diagonal block of $A$ also
lie in a single $G$-orbit. Furthermore,
the matrix $C \in \bp_k( \overline{D}(0,1))$, and
the matrices $A$ and $C$ have equal rank.

\item If the following condition (c) is assumed as well as (a) and (b),
then such a partition $\{ I_1, \ldots, I_k \}$ exists and is unique up
to relabelling of the indices.
\begin{enumerate}
\item[(c)] The diagonal blocks of $A$ have maximal size, i.e., each
diagonal block is not contained in a larger diagonal block that has
rank one.
\end{enumerate}

\item Suppose (a)--(c) hold and $G = \C^\times$. Then the off-diagonal
entries of $C$ lie in the open disc $D( 0, 1 )$.
\end{enumerate}
\end{theorem}

In particular, given any non-zero matrix $A \in \bp_N( \C )$, there
exists a unique partition $\{ I_1, \ldots, I_k \}$ of
$\{ 1, \ldots, N \}$ having minimal size~$k$, unique vectors $\bu_j \in
\C^{| I_j |}$ with $\bu_{j, 1} \in [ 0, \infty )$, and a unique matrix $C
\in \bp_k( \overline{D}( 0, 1 ) )$, such that $A$ is a block matrix with
$A_{I_i \times I_j} = c_{ij} \bu_i \bu_j^*$ whenever
$1 \leq i, j \leq k$, and $A$ has rank at most $k$.

\begin{proof}\hfill
\begin{enumerate}
\item Suppose $1 \leq i \neq j \leq k$, and
$1 \leq l < l' < m \leq N$, with $l$, $l' \in I_i$ and $m \in I_j$;
the submatrix
\[
B := A_{\{ l, l', m \}} = \begin{pmatrix}
a & a g & b \\
a \overline{g} & a | g |^2 & c \\
\overline{b} & \overline{c} & d
\end{pmatrix},
\]
where $a$, $d \geq 0$, $g \in G$, and $b$, $c \in \C$.

We claim that $c \in b \cdot G$, and that the minor
$\begin{pmatrix} a & b \\ a \overline{g} & c \end{pmatrix}$
is singular. Now,
\[
0 \leq \det B = %
-a ( | c |^2 + | b |^2 | g |^2 - 2 \Re( \overline{b} c g ) ) = %
-a | c - b \overline{g} |^2,
\]
so either $a = 0$, in which case $b = c = 0$, by the positivity
of~$B$, or $c = b \overline{g}$. This proves the claim.

Applying this result repeatedly shows that every $2 \times 2$ minor of
the block matrix $A_{I_i \cup I_j}$, with at least two entries in the
same block, is singular, and the entries of any off-diagonal block lie
in a single $G$-orbit. Thus there exists a unique Hermitian matrix $C$
such that all assertions in the first part hold, except for possibly
the claim that $C \in \bp_k( \overline{D}( 0, 1 ) )$.

Fix any vector $\bv \in \C^k$, and choose vectors
$\bw_j \in \C^{| I_j |}$ such that $\bu_j^* \bw_j = v_j$ if
$\bu_j \neq 0$, and arbitrarily otherwise. Define
$\bw := ( \bw_1^*, \ldots, \bw_k^* )^* \in \C^N$, and note that
\begin{equation}\label{Ecompress}
0 \leq \bw^* A \bw = %
\sum_{i, j = 1}^k \bw_i^* c_{i j} \bu_i \bu_j^* \bw_j = %
\sum_{i, j = 1}^k \overline{v_i} c_{i j} v_j = \bv^* C \bv,
\end{equation}
so $C \in \bp_k( \C )$. Consequently, the entries of $C$ are in
$\overline{D}( 0, 1 )$, by a positivity argument, because the diagonal
entries of $C$ are all~$0$ or~$1$.

It remains to show that $A$ and $C$ have the same rank. Let
$J := \{ j \in \{ 1, \ldots, k \} : A_{I_j} \neq 0 \}$ be the set of
indices of non-zero diagonal blocks of $A$, let
$I := \cup_{j \in J} I_j$, and, for all $j \in J$, let $j'$ be the
least element of $I_j$. Define a linear map $\pi : \C^J \to \C^I$ by
letting $\pi( \vi_j ) := \bu_{j', 1}^{-1} \vi'_{j'}$ for all
$j \in J$, where $\vi_j$ and $\vi'_{j'}$ are the
standard basis elements in $\C^J$ and $\C^I$ labelled by $j$ and $j'$,
respectively. Then
\[
\bv^* C_J \bv = \pi( \bv )^* A_I \pi( \bv ), \qquad %
\forall \bv \in \C^J,
\]
so $\pi$ restricts to a linear isomorphism between $\ker C_J$ and
$\ker A_I \cap \pi( \C^J )$. Since the matrices $A$ and $A_I$ have
equal rank, as have $C$ and $C_J$, it follows that the ranks of $A$
and $C$ are equal; note that $\ker A_I \cap \pi( \C^J )$ is naturally
isomorphic to $\ker C_J$, and $\{ \bu'_j : j \in J \}^\perp$ is a
subset of $\ker A_I$ which has trivial intersection with
$\pi( \C^J )$, where $\bu'_j$ is the natural embedding of $\bu_j$ in
$\C^I$.

\item We first establish existence of a partition satisfying (a)--(c),
by induction on~$N$. The result is obvious for $N = 1$, so assume the
result holds for all integers up to and including some $N \geq 1$, and
let $A \in \bp_{N + 1}( \C )$. Let $\{ I_1, \ldots, I_k \}$ be the
partition satisfying properties (a)--(c) for the $N \times N$
upper-left principal submatrix of $A$, and, for each $j$, fix a
non-negative number, denoted by $\alpha_j$, which is a $G$-orbit
representative for all entries in the diagonal block $A_{I_j}$.

Without loss of generality, we may assume $\alpha_j \neq 0$ for
all~$j$, since if $a_{m m} = 0$ then $a_{m n} = a_{n m} = 0$ for
all~$n$, by positivity. We consider three different cases.

\noindent\textbf{Case 1:}
$A_{I_j \cup \{ N + 1 \}} \not\in %
\bp_{| I_j | + 1}^1( \alpha_j \cdot G )$ for all $j$.

In this case, the partition
$\{ I_1, \ldots, I_k, I_{k + 1} := \{ N + 1 \} \}$ clearly has the
desired properties.

\noindent\textbf{Case 2:}
$A_{I_j \cup \{ N + 1 \}} \in \bp_{| I_j | + 1}^1( \alpha_j \cdot G )$
for a single value of $j$.

In this case, augmenting $I_j$ with $N + 1$ yields the desired
partition.

\noindent\textbf{Case 3:}
$A_{I_i \cup \{ N + 1 \}} \in \bp_{| I_i | + 1}^1( \alpha_i \cdot G )$
and
$A_{I_j \cup \{ N + 1 \}} \in \bp_{| I_j | + 1}^1( \alpha_j \cdot G )$
for distinct $i$ and $j$.

We show that this case cannot occur. By the induction hypothesis,
there exists $m \in I_i$ and $n \in I_j$ such that the $2 \times 2$
minor $A_{\{ m, n \}} \not\in \bp_2^1( \alpha_i \cdot G )$. Now set
$a_{N + 1, N + 1} = \alpha$; by assumption, there exist $g$, $h \in G$
such that $a_{m, N + 1} = \alpha g$ and $a_{n, N + 1} = \alpha h$.
Thus
\begin{equation}\label{E3x3blocks}
A_{\{ m, n, N + 1 \}} = \begin{pmatrix}
\alpha | g |^2 & a_{m n} & \alpha g \\
\overline{a_{m n}} & \alpha | h |^2 & \alpha h \\
\alpha \overline{g} & \alpha \overline{h} & \alpha
\end{pmatrix},
\end{equation}
so
\[
0 \leq \det A_{\{ m, n, N + 1 \}} = %
-\alpha | \alpha g \overline{h} - a_{m n} |^2
\]
and therefore $A_{\{ m, n \}} \in \bp_2^1( \alpha_i \cdot G )$, which
is a contradiction.

This completes the inductive step, and existence follows.

To prove the uniqueness of the decomposition, suppose
$\{ I_1, \ldots, I_k \}$ and $\{ J_1, \ldots, J_{k'} \}$ are two
partitions associated to a matrix $A \in \bp_N( \C )$, and satisfying
the desired properties. Without loss of generality, assume
$N \in I_i \cap J_j$, and $I_i \neq J_j$.  Let $\alpha := a_{N, N}$
and note that, since $I_i$ and $J_j$ are distinct and maximal, there
exist $m \in I_i$ and $n \in J_j$ such that $m \neq n$ and
$A_{\{ m, n \}} \not\in \bp_2^1(\alpha_i \cdot G)$ as above. But then
the principal minor $A_{\{ m, n, N \}}$ is of the form
\eqref{E3x3blocks}, which is impossible, as seen previously. It
follows that $I_i = J_j$ and the partition is unique, again by
induction on~$N$.

\item If $| c_{i j} | = 1$ for some $i \neq j$, then
$A_{I_i \cup I_j} = %
\bv \bv^* \in \bp_{| I_i \cup I_j |}^1( \alpha \cdot G )$
for some $\alpha$, where $\bv := ( \bu_i^*, c_{i j} \bu_j^* )^*$.
This contradicts the maximality of $I_i$ and $I_j$, so any
off-diagonal term $c_{i j} \in D( 0, 1 )$.
\qedhere
\end{enumerate}
\end{proof}

\begin{remark}
It is natural to ask if a similar result to Theorem~\ref{Tgroup}
holds if we assume the blocks to have rank bounded above, but not
necessarily by~$1$. This is, however, false, as verified by the
example $A = \Id_3$: the partitions
$\{ \{ 1, 2 \}, \{ 3 \} \}$,
$\{ \{ 2, 3 \}, \{ 1 \} \}$, and
$\{ \{ 1, 3 \}, \{ 2 \} \}$
are all such that $A_{I_i \times I_j}$ has rank at most $2$. However,
$A$ has rank~$3$, so there is no unique maximal partition when we
allow blocks to have rank~$2$ or higher.
\end{remark}

Using the maximal partition corresponding to each
subgroup~$G \subset \C^\times$, we define a stratification of the cone
$\bp_N( \C )$. The following notation will be useful.

\begin{definition}
Fix a subgroup $G \subset \C^\times$ and an integer $N \geq 1$.
\begin{enumerate}
\item Define $( \Pi_N, \prec )$ to be the partially ordered set of
partitions of $\{ 1, \ldots, N \}$, with $\pi' \prec \pi$ if $\pi$ is
a refinement of $\pi'$. Given a partition
$\pi = \{ I_1, \ldots, I_k \} \in \Pi_N$, let $| \pi | := k$ denote
the size of $\pi$.

\item Given a non-zero matrix $A \in \bp_N(\C)$, define
$\pi^G( A ) \in \Pi_N$ to be the unique maximal partition described in
Theorem~\ref{Tgroup}. Also define $\pi^G( {\bf 0}_{N \times N} )$ to
be the indiscrete partition $\{ \{ 1, \ldots, N \} \}$.

\item Given a partition $\pi = \{ I_1, \ldots, I_k \} \in \Pi_N$, let
\begin{equation}
\cals^G_\pi := \{ A \in \bp_N( \C ) : \pi^G( A ) = \pi \}.
\label{Estrata2}
\end{equation}
\end{enumerate}
\end{definition}

Given a subgroup $G$ of $\C^\times$, there is a natural stratification
of the cone according to the structure studied in Theorem
\ref{Tgroup}:
\[
\bp_N( \C ) = \bigsqcup_{\pi \in \Pi_N} \cals^G_\pi,
\]
where the set of strata is in bijection with $\Pi_N$. The following
result discusses basic properties of these Schubert cell-type strata.

\begin{proposition}
Fix a subgroup $G \subset \C^\times$ and an integer $N \geq 1$.
\begin{enumerate}
\item For any partition $\pi \in \Pi_N$, the set $\cals^G_\pi$ has
real dimension $| \pi |^2 + ( N - | \pi | ) \dim_\R G$, and closure
\[
\overline{\cals^G_\pi} = \bigsqcup_{\pi' \prec \pi} \cals^G_{\pi'}.
\]

\item For any $A \in \bp_N( \C )$, the rank of $A$ is at most
$| \pi^{\C^\times}( A ) |$.
\end{enumerate}
\end{proposition}

Note that $G \subset \C^\times$ can have real dimension $0, 1$,
or~$2$.

\begin{proof}\hfill
\begin{enumerate}
\item Suppose $\pi^G( A ) = \{ I_1, \ldots, I_k \}$.
Theorem~\ref{Tgroup} implies that a generic matrix in $\cals^G_\pi$ is
created from a unique matrix $C \in \bp_k( \C^\times )$ with only ones
on the diagonal, and unique non-zero vectors $\bu_j \in \C^{I_j}$ with
$\bu_{j, 1} \in [ 0, \infty )$ and $\bu_{j, l} \in \bu_{j,1} \cdot G$
for all $l > 1$. Thus the degrees of freedom for $A$ equal those for
the strictly upper-triangular entries of~$C$ and for the~$\bu_j$,
i.e.,
\begin{align*}
\dim_\R \cals^G_\pi & = \binom{k}{2} \dim_\R \C + %
\sum_{j = 1}^k ( 1 + ( | I_j | - 1 ) \dim_\R G ) \\
 & = k^2 + ( N - k ) \dim_\R G.
\end{align*}
The second observation is straightforward.

\item This is an immediate consequence of Theorem~\ref{Tgroup}(3).
\qedhere
\end{enumerate}
\end{proof}

The following corollary provides a decomposition of a matrix
$A \in \bp_N( \C )$ that will be very useful for studying the kernel
of $c_0 \one{N} + c_1 A + \cdots + c_{N-1} A^{\circ ( N - 1 )}$.

\begin{corollary}\label{Lgroup}
Let $A \in \bp_N( \C )$ and let $G$ be a multiplicative subgroup of
$S^1$. There is a unique partition $\{ I_1, \ldots, I_k \}$ of
$\{ 1, \ldots, N \}$ such that the corresponding diagonal blocks
$A_{I_j}$ of $A$ satisfy the following properties.
\begin{enumerate}
\item The entries in each diagonal block $A_{I_j}$ belong to
$\alpha_j \cdot G$ for some $\alpha_j \geq 0$.
\item The diagonal blocks have maximal size, i.e., each diagonal block
is not contained in a larger diagonal block with entries in
$\alpha_j \cdot G$ for some $\alpha_j$.
\end{enumerate}
If, moreover, $G = \{ 1 \}$, then $A$ has rank at most $k$.

Finally, if $\{ I_1, \ldots, I_k \}$ is any partition satisfying (1)
but not necessarily (2), then the entries in every off-diagonal block
$A_{I_i \times I_j}$ also share the property that they lie in a single
$G$-orbit in $\C$.
\end{corollary}

Notice that the result follows from Theorem~\ref{Tgroup} because if
$G \subset S^1$, then every block with entries in a single $G$-orbit
automatically has rank at most one, by Theorem~\ref{TMatrix01psd}
below.

In what follows, we use Corollary~\ref{Lgroup} with the following two
choices of the subgroup~$G$:
\begin{enumerate}
\item $G = \{ 1 \}$, in which case all entries in each block of $A$
are equal;
\item $G = S^1$, so all entries in each block of $A$ have equal
modulus.
\end{enumerate}

\begin{remark}
Remark that the diagonal blocks of $A$ as in Corollary~\ref{Lgroup}
may be $1 \times 1$ (for example, in the case where the diagonal
entries of $A$ are all distinct). Moreover, the partition of the
indices as in Corollary~\ref{Lgroup} does not determine the rank
of~$A$. For instance, let $\omega_1$, \ldots, $\omega_N \in S^1$ be
pairwise distinct, and let
$\bu := ( \omega_1, \ldots, \omega_N )^*$. Then the identity matrix
$\Id_N$ and $\bu \bu^*$ have different ranks for $N \geq 2$, but both
matrices correspond to the partition of $\{ 1, \ldots, N \}$ into
singleton subsets.
\end{remark}

\subsection{Simultaneous kernels of Hadamard powers}

We now state and prove the main result of this section, which in
particular classifies the simultaneous kernels of Hadamard powers of a
positive semidefinite matrix.

\begin{theorem}\label{Tsimult}
Let $A \in \bp_N( \C )$ and let $\{ I_1, \ldots, I_k \}$ be the unique
partition of $\{ 1, \ldots, N \}$ satisfying the two conditions of
Corollary~\ref{Lgroup} with $G = \{ 1 \}$. Fix $B \in \bp_N( \C )$ with
no zero diagonal entries, and let $c_0$, \ldots, $c_{N - 1} > 0$. Then
\begin{equation*}
\ker \bigl( B \circ %
( c_0 \one{N} + \cdots + c_{N - 1} A^{\circ ( N - 1 )} ) \bigr) = %
\bigcap_{n \geq 0} \ker ( B \circ A^{\circ n} ) = %
\ker B_{I_1} \oplus \cdots \oplus \ker B_{I_k},
\end{equation*}
where $B_{I_j}$ are the diagonal blocks of $B$ corresponding to the
partition $\{ I_1, \ldots, I_k \}$.
\end{theorem}

\noindent
Note that when the Hadamard product is replaced by the standard matrix
product, the simultaneous kernel $\cap_{n \geq 1} \ker A^n$ equals
$\ker A$. In contrast, characterizing the simultaneous kernels of
Hadamard powers is a challenging problem.  Also observe in
Theorem~\ref{Tsimult} that the simultaneous kernel does not depend on
$c_0$, \ldots, $c_{N - 1}$.

The proof of Theorem~\ref{Tsimult} repeatedly uses a technical result,
which we quote here for convenience.

\begin{theorem}[{Hershkowitz--Neumann--Schneider,
\cite[Theorem~2.2]{Matrix01psd}}]\label{TMatrix01psd}
Given an $N \times N$ complex matrix $A$, where $N \geq 1$, the
following are equivalent.
\begin{enumerate}
\item $A$ is positive semidefinite with entries of modulus $0$ or $1$,
i.e., $A \in \bp_N( S^1 \cup \{ 0 \} )$.
\item There exist a diagonal matrix $D$, all of whose diagonal entries
lie in $S^1$, as well as a permutation matrix $Q$, such that
$( Q D )^{-1} A ( Q D )$ is a block diagonal matrix with each diagonal
block a square matrix of either all ones or all zeros.
\end{enumerate}
\end{theorem}

Equipped with this result, we now prove the above theorem.

\begin{proof}[Proof of Theorem~\ref{Tsimult}]
We begin by showing the first equality. One inclusion is immediate;
for the reverse inclusion, let
$\bu \in \ker \bigl( B \circ ( c_0 \one{N} + c_1 A + \cdots + %
c_{N - 1} A^{\circ ( N - 1 )} ) \bigr)$.
Then $\bu \in \ker ( B \circ A^{\circ n} )$ for $0 \leq n \leq N - 1$,
since $A$ and $B$ are positive semidefinite.
Applying Lemma~\ref{Lalgebra}, we conclude that
$\bu \in \bigcap_{n \geq 0} \ker ( B \circ A^{\circ n} )$.

We now prove the second equality. Let $\{ J_1, \ldots, J_l \}$ be a
partition of $\{ 1, \ldots, N \}$ as in Corollary~\ref{Lgroup} with
$G = S^1$, i.e., with the entries in the diagonal blocks having the same
absolute value, instead of necessarily being constant. Clearly,
$\{ I_1, \ldots, I_k \}$ is a refinement of the partition
$\{ J_1, \ldots, J_l \}$. We proceed in three steps. We first show
that
\begin{equation}\label{Estep1}
\bigcap_{n \geq 0} \ker ( B \circ A^{\circ n} ) \subset %
\bigcap_{m \geq 1} \ker( B_{J_1} \circ A_{J_1}^{\circ m} ) \oplus %
\cdots \oplus \bigcap_{m \geq 1} %
\ker( B_{J_l} \circ A_{J_l}^{\circ m} ).
\end{equation}
We then prove that each kernel of the form
$\bigcap_{m \geq 1} \ker( B_{J_i} \circ A_{J_i}^{\circ m} )$ further
decomposes into
$\ker B_{I_{i_1}} \oplus \cdots \oplus \ker B_{I_{i_p}}$,
where $J_i = I_{i_1} \cup \cdots \cup I_{i_p}$. Finally,
we prove the reverse inclusion, i.e.,
\[
\bigoplus_{m = 1}^k \ker B_{I_m} \subset %
\bigcap_{n \geq 0} \ker( B \circ A^{\circ n} ).
\]

Equation~\eqref{Estep1} is obvious if $l = 1$, so assume $l \geq 2$.
Let $i$ satisfy $a_{i i} \geq \max_{j =1}^N a_{j j} > 0$ and suppose,
without loss of generality, that $i \in J_1$. Now write the matrices
$A$ and $B$ in block form:
\[
A = \begin{pmatrix}
 A_{1 1} & A_{1 2} \\
 A_{1 2}^* & A_{2 2}
\end{pmatrix} \quad \text{and} \quad %
B = \begin{pmatrix}
 B_{1 1} & B_{1 2} \\
 B_{1 2}^* & B_{2 2}
\end{pmatrix},
\]
where the $( 1, 1 )$ blocks correspond to the $J_1 \times J_1$
entries, and the $( 1, 2 )$ blocks correspond to the
$J_1 \times J_1^c$ entries of the matrices, where
$J_1^c := J_2 \cup \cdots \cup J_l$.  Then, by
Theorem~\ref{TMatrix01psd}, we conclude that
$a_{i i}^{-1} A_{1 1} = \bv \bv^*$ for some
$\bv \in ( S^1 )^{| J_1 |}$. Moreover, we claim that the entries of
$a_{i i}^{-1} A_{1 2}$ have modulus less than one. Indeed, since the
entries of the matrix $a_{i i}^{-1} A$ lie in the closed unit disc
$\overline{D}( 0, 1 )$, by a positivity argument, we can choose a
sequence of integer powers $n'_k \to \infty$ such that
$A_\infty := \lim_{k \to \infty} ( a_{i i}^{-1} A)^{\circ n'_k}$
exists entrywise. Note that $A_\infty \in \bp_N( \C )$, by the Schur
product theorem. Now let $m \in J_1^c$ and consider the submatrix
$A' := ( A_\infty )_{J_1 \cup \{ m \}}$, which is positive
semidefinite and has entries with modulus $0$ or $1$. Thus, by
Theorem~\ref{TMatrix01psd}, $A' = \bw \bw^*$ for some
$\bw \in \C^{| J_1 | + 1}$ with $| w_i | = 0$ or $1$.  By the
maximality of $J_1$, it follows that $| a_{i i}^{-1} a_{j, m} | < 1$
for all $j \in J_1$ and so all the entries of $a_{i i}^{-1} A_{1 2}$
have modulus less than $1$.

Now let
$\bu = ( \bu_1^*, \bu_2^* )^* \in %
\bigcap_{n \geq 0} \ker ( B \circ A^{\circ n} )$,
where $\bu_1 \in \C^{| J_1 |}$, $\bu_2 \in \C^{| J_1^c |}$, and
$\| \bu \| = 1$. We will prove that
$( B_{1 1} \circ A_{1 1}^{\circ n} ) \bu_1 = 0$ and
$( B_{2 2} \circ A_{2 2}^{\circ n} ) \bu_2 = 0$ for all $n \geq 1$. To
do so, fix $\epsilon \in ( 0, 1 )$ and let
$\alpha_j := \arg( v_j ) / ( 2 \pi )$ for $j = 1$, \ldots, $| J_1 |$.
Suppose $\{ \theta_0 := 1, \theta_1, \ldots, \theta_p \}$ is a
$\Q$-linearly independent basis of the $\Q$-linear span of
$\{1, \alpha_1, \ldots, \alpha_{| J_1 |} \}$. Then there exist
integers $m_{j k}$, and an integer $M \geq 1$, such that
\[
\alpha_j = \frac{1}{M} \sum_{k = 0}^p m_{j k} \theta_k %
\qquad ( j = 1, \ldots, | J_1 | ).
\]
By Kronecker's theorem \cite[Chapter~23]{HardyWright}, since
$\theta_k / M$ are $\Q$-linearly independent, there exist
sequences of integers $( n_r )_{r = 1}^\infty$ and
$( p_r )_{r = 1}^\infty$ such that $n_r \to \infty$ and
$| n_r ( \theta_k / M ) - ( \theta_k / M ) - p_r | < \epsilon$ for
$0 \leq k \leq p$. It follows that, for any $j$, $j' \in J_1$,
\begin{align*}
| ( a_{i i}^{-1} A_{1 1} )_{j, j'}^{n_r} - %
( a_{i i}^{-1} A_{1 1} )_{j, j'} | & = %
|\exp\bigl( 2 \pi i n_r ( \alpha_j - \alpha_{j'} ) \bigr) - %
\exp\bigl( 2 \pi i ( \alpha_j - \alpha_{j'} ) \bigr) | \\
 & \leq \biggl| 2 \pi \sum_{k = 0}^p %
( m_{j k} - m_{j' k} ) \Bigl( \frac{n_r \theta_k - \theta_k}{M} - %
p_r \Bigr) \biggr| \\
 & \leq 2 \pi \max_{k = 0, \ldots, p} %
| n_r ( \theta_k / M ) - ( \theta_k / M ) - p_r | %
\sum_{k = 0}^p | m_{j k} - m_{j' k} | \\
 & \leq C_{j j'} \epsilon
\end{align*}
for some constant $C_{j j'} \geq 0$ independent of $\epsilon$. Let
$C := \max\{ 1, C_{j j'} : j, j' \in J_1 \}$.

Replacing $( n_r )_{r = 1}^\infty$ by a subsequence if necessary, we
may assume without loss of generality that
$( a_{i i}^{-1} A )^{\circ n_r}$ converges entrywise to a limit
\[
A_\infty := \begin{pmatrix}
 A_{\infty, 1 1} & {\bf 0}_{| J_1 | \times | J_1^c |} \\
 {\bf 0}_{| J_1^c | \times | J_1 | } & A_{\infty, 2 2}
\end{pmatrix}.
\]
Passing to a further subsequence, we may also assume the entries of
$( a_{i i}^{-1} A_{1 1} )^{\circ n_r} - A_{\infty, 1 1}$ are at most
$C \epsilon$ in modulus.  Moreover,
\[
\sum_{k = 1}^2 \bu_k^* ( B_{k k} \circ A_{\infty, k k} ) \bu_k = %
\lim_{r \to \infty} \bu^* ( B \circ ( a_{i i}^{-1} A )^{\circ n_r} ) \bu = 0,
\]
whence $( B_{k k} \circ A_{\infty, k k} ) \bu_k = 0$ for $k = 1$, $2$.
Therefore, since $\| \bu_1 \| \leq 1$, 
\begin{align*}
\| ( B_{1 1} & \circ A_{1 1} ) \bu_1 \| \\
 & \leq a_{i i} \| \bigl( B_{1 1} \circ ( a_{i i}^{-1} A_{1 1} - %
( a_{i i}^{-1} A_{1 1} )^{\circ n_r} ) \bigr) \bu_1 \| %
+ \| ( B_{1 1} \circ \bigl( ( a_{i i}^{-1} A_{1 1} )^{\circ n_r} - %
A_{\infty, 1 1 } \bigr) ) \bu_1 \| \\
 & \leq a_{i i} \| B_{1 1} \circ ( a_{i i}^{-1} A_{1 1} - %
( a_{i i}^{-1} A_{1 1} )^{\circ n_r} ) \| + %
\| B_{1 1} \circ \bigl( ( a_{i i}^{-1} A_{1 1} )^{\circ n_r} - %
A_{\infty, 1 1} \bigr) \| \\
 & \leq \max_{j, k \in J_1} | B_{j k} | \, %
( a_{i i} + 1 ) \, | J_1 | \, C \epsilon;
\end{align*}
as $\epsilon$ is arbitrary, we must have 
$( B_{1 1} \circ A_{11} ) \bu_1 = 0$. Furthermore, since
$\bu \in \ker B \circ A$, so
\[
( B_{1 1} \circ A_{1 1}) \bu_1 + ( B_{1 2} \circ A_{1 2} ) \bu_2 = 0,
\]
whence $( B_{1 2} \circ A_{1 2} ) \bu_2 = 0$, and
\[
( B_{1 2} \circ A_{1 2} )^* \bu_1 + %
( B_{2 2} \circ A_{2 2} ) \bu_2 = ( B_{2 1} \circ A_{2 1} ) \bu_1 + %
( B_{2 2} \circ A_{2 2} ) \bu_2 = 0.
\]
This implies that
\[
0 = \bigl( ( B_{1 2} \circ A_{1 2} ) \bu_2 )^* \bu_1 + %
\bu_2^* ( B_{2 2} \circ A_{2 2} ) \bu_2 = 
\bu_2^* ( B_{2 2} \circ A_{2 2} ) \bu_2,
\]
and therefore $( B_{2 2} \circ A_{2 2} ) \bu_2 = 0$, since
$B_{2 2} \circ A_{2 2}$ is positive semidefinite.

Repeating the same argument, with $A$ replaced by $A^{\circ m}$ for
some fixed $m \geq 1$, we conclude that
\[
(B_{k k} \circ A_{k k}^{\circ m}) \bu_k = 0 \qquad %
( k = 1, 2, \ m \geq 1 ).
\]
Hence
\[
\bigcap_{m \geq 1} \ker( B \circ A^{\circ m} ) \subset %
\bigcap_{m \geq 1} \ker( B_{1 1} \circ A_{1 1}^{\circ m} ) \oplus %
\bigcap_{m \geq 1} \ker( B_{2 2} \circ A_{2 2}^{\circ m} ),
\]
and we conclude by induction that
\[
\bigcap_{n \geq 0} \ker( B \circ A^{\circ n} ) \subset %
\bigcap_{m \geq 1} \ker( B_{J_1} \circ A_{J_1}^{\circ m} ) \oplus %
\cdots \oplus %
\bigcap_{m \geq 1} \ker( B_{J_l} \circ A_{J_l}^{\circ m} ).
\]

We now examine the simultaneous kernel of Hadamard powers of a
non-zero diagonal block $B_{J_i} \circ A_{J_i}^{\circ m}$. Assume
without loss of generality that $| a_{j k} | = 1$ for all $j$,
$k \in J_i$, and that $J_i = I_1 \cup \cdots \cup I_t$ for some
integer $t \geq 1$. By Theorem~\ref{TMatrix01psd}, we obtain
$A_{J_i} = \bv \bv^*$ for some $\bv \in ( S^1 )^{| J_i |}$. It follows
that $A_{J_i}$ is itself a block matrix with $1$ throughout each
diagonal sub-block.  Thus we can write
\[
\bv = ( \lambda_1 {\bf 1}_{1 \times n_1}, \ldots, %
\lambda_t {\bf 1}_{1 \times n_t} )^T \quad \text{and} \quad %
A_{J_i} = \bv \bv^* = %
( \lambda_i \overline{\lambda_j} %
{\bf 1}_{n_i \times n_j} )_{i, j = 1}^t,
\]
with $\lambda_1$, \ldots, $\lambda_t \in S^1$ pairwise distinct, and
$| J_1 | = n_1 + \cdots + n_t$.

Now let
$\bu \in \bigcap_{m \geq 1} \ker ( B_{J_i} \circ A_{J_i}^{\circ m} )$
and let
$\bu = ( \bu_1^*, \ldots, \bu_t^* )^* \in %
\C^{n_1} \oplus \cdots \oplus \C^{n_t}$
be the decomposition of $\bu$ corresponding to the
partition $\{ I_1, \ldots, I_t \}$ of $J_i$. Let
$B_{j k} := B_{I_j, I_k}$; we claim that $\bu_j \in \ker B_{j j}$ for
all $j$. Note that
\[
\sum_{k = 1}^t B_{j k} ( \lambda_j \overline{\lambda_k} )^m \bu_k %
= 0 \quad \forall m \geq 1,
\]
from which it follows that
\[
\sum_{k = 1}^t ( \bu_j^* B_{j k} \bu_k ) %
\bigl( \overline{\lambda_k} \bigr)^{m - 1} = 0 %
\quad \forall m \geq 1.
\]
Thus, for fixed $j$, the vector $( \bu_j^* B_{j k} \bu_k )_{k = 1}^t$
belongs to the kernel of the transpose of the Vandermonde matrix
\begin{equation*}
V = ( \overline{\lambda_j}^{k - 1} )_{j, k =1}^t.
\end{equation*}
Since the $\lambda_j$ are distinct and non-zero, the matrix $V$ is
non-singular. Consequently, $\bu_j^* B_{j k} \bu_k = 0$ for all $j$,
$k$; in particular, $\bu_j^* B_{j j} \bu_j = 0$, and therefore
$\bu_j \in \ker B_{j j}$.  This completes this step and shows that
\[
\bigcap_{n \geq 0} \ker ( B \circ A^{\circ n} ) \subset %
\ker B_{I_1} \oplus \cdots \oplus \ker B_{I_t}.
\]

We now prove the reverse inclusion. First, we claim that if
$C = ( C_{i j} )_{i, j = 1}^t$ is a block matrix in $\bp_N( \C )$, and
$\bu = ( \bu_1^*, \ldots, \bu_t^* )^* \in \C^N$ is the corresponding
block decomposition of $\bu$, then
\begin{equation}\label{Equad}
\bu \in \ker C_{1 1} \oplus \cdots \oplus \ker C_{t t} %
\quad \implies \quad \bu \in \ker C.
\end{equation}
We prove the claim by induction on $t$. The base case of $t = 1$ is
obvious. Now suppose \eqref{Equad} holds for $t - 1$ blocks. Let
$\bu = ( \bu_1^*, \ldots, \bu_t^* )^* = ( {\bu'}^*, \bu_t^*)^*$,
with $\bu_j \in \ker C_{j j}$.  Then $\bu_j^* C_{j j} \bu_j = 0$ for
all $j$. Partition $C$ with respect to the same decomposition:
\begin{equation*}
C = \begin{pmatrix} C' & C'_t \\ ( C'_t )^* & C_{t t} \end{pmatrix}.
\end{equation*}
By the induction hypothesis, ${\bu'}^* C' \bu' = 0$. On the other
hand, for any $\lambda \in \R$,
\begin{equation*}
0 \leq %
( {\bu'}^*, \lambda \bu_t^* ) C ( {\bu'}^*, \lambda \bu_t^* )^* = %
2 \lambda \Re ( {\bu'}^* C'_t \bu_t ).
\end{equation*}
This implies that $\Re ( {\bu'}^* C'_t \bu_t ) = 0$, from which it
follows immediately that $\bu^* C \bu = 0$, and so $\bu \in \ker C$.
This proves the claim.

Now, to conclude the proof, let
$\bu \in \ker B_{I_1} \oplus \cdots \oplus \ker B_{I_k}$. Then, for
all $n \geq 0$,
$\bu \in \ker (B_{I_1} \circ A_{I_1} ^{\circ n}) \oplus \cdots %
\oplus \ker ( B_{I_k} \circ A_{I_k}^{\circ n} )$,
since the entries of $A$ are constant in each diagonal block
$A_{I_j}$. It follows by~\eqref{Equad} that
$\bu \in \ker ( B \circ A^{\circ n} )$ for all $n \geq 0$.

This concludes the proof of the theorem.
\end{proof}

\begin{remark}
We note the following consequence of Theorem~\ref{Tsimult}.
Given a matrix $B \in \bp_N( \C )$ with no zero diagonal entries, as
$A$ runs over the uncountable set $\bp_N( \C )$, the set of
simultaneous kernels
\[
\{ \cap_{n \geq 0} \ker ( B \circ A^{\circ n} ) : A \in \bp_N( \C ) \}
\]
is, nevertheless, a finite set of subspaces of $\C^N$. Moreover, this
finite set is indexed by partitions of the set
$\{ 1, \ldots, N \}$. The case when $B = \one{N}$, or, more generally,
when $B$ has no zero diagonal entries, is once again in contrast with
the behaviour for the usual matrix powers, and provides a
stratification of the cone~$\bp_N( \C )$.
\end{remark}

We conclude this section by strengthening Theorem~\ref{Tsimult}. Given
an integer $N \geq 1$, matrices $A$, $B \in \bp_N( \C )$, and a
partition $\pi = \{ I_1, \ldots, I_k \} \in \Pi_N$, let
\begin{equation}
\mathcal{K}_\pi( A, B ) := \bigoplus_{j = 1}^k %
\bigcap_{n \geq 0} \ker( B_{I_j} \circ A_{I_j}^{\circ n} ).
\end{equation}
We have from the proof of Theorem~\ref{Tsimult} that
\begin{equation}\label{Esimult}
\mathcal{K}_{\{ \{ 1, \ldots, N \} \}}( A, B ) = %
\mathcal{K}_{\pi^{S^1}( A )}( A, B ) = %
\mathcal{K}_{\pi^{\{ 1 \}}( A )}( A, B ).
\end{equation}
Our final result analyzes the set of partitions $\pi$ for which
Equation~\eqref{Esimult} holds.

\begin{theorem}
Fix an integer $N \geq 1$ and matrices $A$, $B \in \bp_N( \C )$, with
$B$ having non-zero diagonal entries. Then,
\begin{equation}
\{ \pi \in \Pi_N : \mathcal{K}_{\{ \{ 1, \ldots, N \} \}}( A, B ) = %
\mathcal{K}_\pi( A, B ) \} \supset %
\{ \pi : \pi \prec \pi^{\{ 1 \}}( A ) \}.
\end{equation}
The reverse inclusion holds if $B \in \bp_N^1( \C )$.
\end{theorem}

In particular, for any subgroup $G \subset \C^\times$,
Equation~\eqref{Esimult} holds with $\pi^{S^1}( A )$ replaced by
$\pi^G( A )$.

\begin{proof}
Suppose $\pi^{\{ 1 \}}( A ) = \{ I_1, \ldots, I_k \}$ is a refinement
of $\pi = \{ J_1, \ldots, J_l \}$. Then, restricting to each
$J_p \times J_p$ diagonal block,
\[
\mathcal{K}_{\pi|_{J_p}}( A_{J_p}, B_{J_p} ) := %
\bigcap_{n \geq 0} \ker( B_{J_p} \circ A_{J_p}^{\circ n} )
\]
equals $\mathcal{K}_{\pi^{\{ 1 \}}(A_{J_p})}( A_{J_p}, B_{J_p})$, by
Theorem~\ref{Tsimult}. We are now done by taking the direct sum of the
previous equation over all $p$, since
$\bigsqcup_{p = 1}^l \pi^{\{ 1 \}}(A_{J_p}) = \pi^{\{ 1 \}}( A )$.

Conversely, suppose that $B \in \bp_N^1( \C )$ and that
$\pi^{\{ 1 \}}( A)$ is not a refinement of $\pi$. Then, without loss
of generality, there exist indices $i_1$ and $i_2$ which lie in
distinct parts of $\pi$ but the same part of $\pi^{\{ 1 \}}( A )$.
The vector with $i_1$th entry equal to $b_{i_1, i_2}$, its $i_2$th
entry equal to $-b_{i_1, i_1}$ and all other entries equal to $0$,
lies in $\mathcal{K}_{\pi^{\{ 1 \}}( A )}( A, B )$, but it does not
lie in $\mathcal{K}_\pi( A, B )$.
\end{proof}

\section{Conclusion and survey of known results}

It is the aim of the present section to discuss, from a unifying point
of view, a collection of old and new computations of sharp bounds for
extreme critical values of certain matrix pencils. As mentioned in the
introduction, this was a recurrent theme, motivated by theoretical and
very applied problems, spanning more than half a century.

In all the examples which follow, we identify a numerical evaluation
of the extreme critical value of a concrete matrix pencil.
An authoritative source for the spectral theory of polynomial pencils of
matrices is~\cite{Markus}.

Specifically, Equation~\eqref{Eschoenberg1} below provides an
accessible, often computationally effective, way of expressing the
rather elusive $\mathcal{C}( h; g; \bp )$, which is, by definition,
the smallest real constant $C$ satisfying
\[
g[ A ] \leq C h[ A ], \qquad \mbox{for all matrices } A \in \bp.
\]
We do not exclude above the case $C = \infty$, which means that no
uniform bound between $g[ A ]$ and $h[ A ]$ exists. A second general
observation is the stability of the bound as a function of the matrix
set: more precisely, quite a few examples below share the property
\[
\mathcal{C}( h; g; \bp ) = \mathcal{C}( h; g; \bp' ),
\]
where $\bp' \subset \bp$ is a much smaller class of matrices.

\begin{enumerate}
\item The first of Schoenberg's celebrated theorems proved in
\cite{Schoenberg42} involves convergent series in Gegenbauer
polynomials. The result can be formulated as a matrix pencil
critical-value problem, as follows: fix an integer $d \geq 2$, set
$K := [ {-1}, 1 ]$, and define
\[
h( z ) = \sum_{n \geq 0} h_n C_n^{( \lambda )}( z ) %
\qquad \text{and} \qquad %
g( z ) = \sum_{n \geq 0} g_n C_n^{( \lambda )}( z ),
\]
where $h_n \in [ 0, \infty )$, $g_n \in \R$,
$\lambda = ( d - 2 ) / 2$, and $C_n^{( \lambda )}$ is the
corresponding Gegenbauer or Chebyshev polynomial. Also let $\bp$
denote the set of all correlation matrices with rank at most $d$ but
of arbitrary dimension. Then Schoenberg's Theorem~\ref{Tspheres}(1)
asserts that
\begin{equation}\label{Eschoenberg1}
\mathcal{C}( h; g; \bp ) = \sup_{n : g_n > 0} \frac{g_n}{h_n}.
\end{equation}
When $h_n = 0$ and $g_n > 0$ for some index $n$, the constant
$\mathcal{C}( h; g; \bp )$ is equal to infinity, whence there is no
uniform bound for $g[ A ]$ in terms of $h[ A ]$, when taken over all
matrices $A \in \bp$.

\item Schoenberg's second landmark result from \cite{Schoenberg42}, as
well as its subsequent extensions by Christensen and Ressel, Hiai, and
others (see the references after Theorem~\ref{Tspheres}), can also be
rephrased as an extreme critical-value problem, as follows. Set
$K := ( {-\rho}, \rho )$, where $0 < \rho \leq \infty$, and consider
the convergent power series
\begin{equation}\label{Eschoenberg2}
h( z ) = \sum_{n \geq 0} h_n z^n, \ %
g( z ) = \sum_{n \geq 0} g_n z^n : ( {-\rho}, \rho ) \to \R,
\end{equation}
where $h_n \in [ 0, \infty )$ and $g_n \in \R$ for all $n$, and
$\bp = \bigcup_{N \geq 1} \bp_N( ( {-\rho}, \rho ) )$. Then
$\mathcal{C}( h; g; \bp )$ can be computed as in
Equation~\eqref{Eschoenberg1}.

\item In the papers \cite{Kahane-Rudin,Rudin59}, Rudin, working with
Kahane, proved that preserving positivity on low-rank Toeplitz
matrices already implies absolute monotonicity. In this case, set
$K := ( {-1}, 1 )$, and let $h( z )$ and $g( z )$ to be as in
Equation~\eqref{Eschoenberg2} with $\rho = 1$. Also let 
$\bp := \bigcup_{N \geq 1} \bp_N( ( {-\rho}, \rho) )$ and let
$\bp'$ denote the set of Toeplitz matrices  of all dimensions and of
rank at most~$3$. Rudin showed in \cite[Theorem~IV]{Rudin59} that it
suffices to test the pencil bound on the set
$\bp_\alpha := \{ M( a, b , n, \alpha ) : 0 \leq a, b, \ a + b < 1,
n \geq 1 \}$
for any irrational multiple~$\alpha$ of~$\pi$, where
$M( a, b, n, \alpha ) \in \bp_n^3( [ {-1}, 1 ] )$ is the Toeplitz
matrix with $( j, k )$th entry $a + b \cos( ( j - k ) \alpha)$.

Then
$\mathcal{C}( h; g; \bp ) = \mathcal{C}( h; g; \bp' ) = %
\mathcal{C}( h; g; \bp_\alpha )$
and Equation~\eqref{Eschoenberg1} holds.

\item A necessary condition for preserving positivity  in fixed
dimension was provided by Horn in \cite{horn}. The condition was
subsequently extended by Guillot--Khare--Rajaratnam \cite{GKR-lowrank}
and is stated in Theorem~\ref{Thorn} above. This yields a special case
of the extreme critical-value problem, with
\begin{equation}\label{Eschoenberg3}
h( z ) = \sum_{n = 0}^{N - 1} h_n z^n \quad \text{and} \quad %
g( z ) = \sum_{n = 0}^{N - 1} g_n z^n,
\end{equation}
where $h_n$, $g_n \in [ 0, \infty )$ for all $n$, and
$\bp = \bp_N^2( ( 0, \rho ) )$. Then
$\mathcal{C}( h; g; \bp ) = %
\mathcal{C}( h; g; \bp_N( ( 0, \rho ) ) )$
and Equation~\eqref{Eschoenberg1} holds.

\item The problem of preserving positivity has recently attracted
renewed attention, due to its application in the regularization of
ultra high-dimensional covariance matrices. The next few observations
are along those lines. First, in
\cite[Proposition~3.17(3)]{GKR-lowrank}, the authors consider a more
general situation than the previous instance, where one replaces the
polynomials $h$, $g$ of degree at most $N - 1$ by one of the following.
\begin{itemize}
\item A linear combination of $N$ fractional powers
$z^\alpha$, with $K = ( 0, \rho )$ and $\bp = \bp_N^1( K )$; here
$\alpha$ can be negative.

\item A linear combination of $N$ fractional powers
$z^\alpha$ and the constant function $1$, with $K = [ 0, \rho )$ and
$\bp = \bp_N^1( K )$; here $0^\alpha := 0$ for $\alpha \in \R$.

\item A linear combination of $N$ fractional powers of the form
$\phi_\alpha( z ) := | z |^\alpha$ or
$\psi_\alpha( z ) := \sgn( z ) | z |^\alpha$, with
$K = ( {-\rho}, \rho )$ and $\bp = \bp_N^1( K )$; here
$\phi_\alpha( 0 ) = \psi_\alpha( 0 ) := 0$.
\end{itemize}
In each case, if $h_n$, $g_n \in [ 0, \infty )$ for all $n$, with both
$h$ and $g$ involving the same set of fractional powers, then
$\mathcal{C}( h; g; \bp ) = \mathcal{C}( h; g; \bp_N( K ) )$ and
Equation~\eqref{Eschoenberg1} holds.

\item Note that Lemma~\ref{Lpos-coeff} also yields an extreme critical
value that can be deduced from either of the two previous cases. In
particular, if $h$ and $g$ are power series, as in
Equation~\eqref{Eschoenberg2}, and
we define 
$\bp = \bigcup_{N \geq 1} \bp_N( ( 0, \rho ) )$  and
$\bp' = \bigcup_{N \geq 1} \bp_N^1( ( 0, \rho ) )$, then
$\mathcal{C}( h; g; \bp ) = \mathcal{C}( h; g; \bp')$ and
Equation~\eqref{Eschoenberg1} holds. This result was also shown, using
alternate approaches, in \cite{GKR-lowrank}.

\item Another application involves entrywise functions preserving
positivity on matrices with zero structure according to a tree graph
$T$. Recall that for a graph $G$ with vertex set $\{ 1, \ldots, N \}$,
and a subset $K \subset \C$, the cone $\bp_G( K )$ is defined to be
the set of matrices $A = ( a_{i j} ) \in \bp_N( K )$ such that if
$i \neq j$ and $( i, j )$ is not an edge in $G$, then $a_{i j} = 0$.
An extreme critical-value phenomenon was shown in
\cite{Guillot_Khare_Rajaratnam2012}: fix powers
\[
0 < r' < r < s < s' < \infty, \quad r > 1,
\]
as well as a measurable set $B \subset ( r, s )$. Also fix scalars
$a_{r'}$, $a_r$, $a_s$, $a_{s'} > 0$ and a measurable function
$l : B \to \R$. Now define
\[
h( z ) = a_{r'} z^{r'} + a_r z^r + a_s z^s + a_{s'} z^{s'}, \qquad
g( z ) = \int_B l( z ) \std z \qquad ( z \in \R ),
\]
and $\bp = \bigcup_{T \in \cals} \bp_T( ( 0, \infty ) )$, where
$\cals$ is a non-empty set of connected trees on at least $3$
vertices. Then \cite[Theorem~4.6]{Guillot_Khare_Rajaratnam2012} shows
the existence of a  finite threshold:
$\mathcal{C}( h; g; \bp ) \in ( 0, \infty )$. Note that this is an
existence result, in contrast to the sharp bounds obtained in all
previous examples.

\item We turn now to the present paper, in which extreme critical
values were obtained for various families of linear pencils. For all
of the remaining examples, fix the following notation:
$N \geq 1$ and $M \geq 0$ are integers, $\rho \in ( 0, \infty )$, and
\begin{align}\label{Enotation}
c_0, c_1, \ldots, c_{N - 1} \in ( 0, \infty ), \notag \\
h_\bc( z ) = c_0 + \cdots + c_{N - 1} z^{N - 1}, \notag \\
\text{and} \quad \bp_N^1( ( 0, \rho ) ) \subset \bp \subset %
\bp_N( \overline{D}( 0, \rho ) ).
\end{align}
Then the main result of the present paper, Theorem~\ref{Tthreshold},
says that
\[
\mathcal{C}( h_\bc; z^M; \bp ) = %
\sum_{j = 0}^{N - 1} \binom{M}{j}^2 \binom{M - j - 1}{N - j - 1}^2 %
\frac{\rho^{M - j}}{c_j}.
\]
Note that if $M = N$, $\bp = \bp_N( [ {-\rho}, \rho ] )$, and $N$
tends to infinity, then the extreme critical value grows without
bound, thereby recovering Schoenberg's result discussed in the second
example above.

\item The $\R_+$-subadditivity of $\mathcal{C}( h; g; \bp )$ in its
second argument, as in Equation~\eqref{Esubadditive}, immediately
yields, for a polynomial $g( z ) = \sum_{n = 0}^M g_n z^n$, that
\[
\mathcal{C}( h_\bc; g; \bp ) \leq %
\sum_{n : g_n > 0} g_n \mathcal{C}( h_\bc; z^n; \bp ),
\]
where the notation is as in~\eqref{Enotation}.

\item When $g( z ) = \sum_{M = N}^\infty c_M z^M$ is analytic on
$D( 0, \rho )$ and continuous on $\overline{D( 0, \rho )}$, with real
coefficients, Theorem~\ref{Tanalytic} yields a bound on the corresponding
extreme critical value:
\[
\mathcal{C}( \bc; g; N, \rho ) \leq %
\frac{g_2^{( 2 N - 2 )}( \sqrt{\rho} )}{2^{N - 1} ( N - 1)!^2} %
\sum_{j = 0}^{N - 1} \binom{N - 1}{j}^2 \frac{\rho^{N - j - 1}}{c_j},
\]
where $g_+( z ) = \sum_{M \geq N : c_M > 0} c_M z^M$ and
$g_2( z ) = g_+( z^2 )$.

\item In the case of $2 \times 2$ matrices, fix non-negative integers
$m < n < p$, set $h_2( z ) = c_m z^m + c_n z^n$, and suppose
$\bp_2^1( [ 0, 1 ] ) \subset \bp' \subset \bp_2( [ 0, 1 ] )$. Then
Theorem~\ref{T2x2} shows that
\[
\mathcal{C}( h_2; z^p; \bp') = %
\frac{c_m c_n ( n - m )^2}{c_m ( p - m )^2 + c_n ( p - n )^2}.
\]

\item The next instance involves Rayleigh quotients for Hadamard
powers. In this case, we let $\bp$ be the set containing a single
non-zero matrix $A \in \bp_N( \C )$. Define
\[
\mathcal{K}( A ) = \ker h_\bc[ A ] = %
\ker( c_0 \one{N} + c_1 A + \cdots + c_{N - 1} A^{\circ ( N - 1 )} ).
\]
Then Propositions~\ref{Palgebra} and \ref{Ppos-constant} show that,
with notation as in~\eqref{Enotation},
\[
\mathcal{C}( h_\bc; z^M; A ) = %
\max_{\bv \in S^{2 N - 1} \cap \mathcal{K}( A )^\perp} %
\frac{\bv^* A^{\circ M} \bv}{\bv^* h_\bc[ A ] \bv} = %
\varrho( h_\bc[ A ]^{\dagger / 2} A^{\circ M} h_\bc[ A ]^{\dagger / 2} ) %
< \infty.
\]
In particular,
$\mathcal{C}( h_\bc; z^M; \bu \bu^* ) = %
( \bu^{\circ M} )^* h_\bc[ \bu \bu^* ]^\dagger \bu^{\circ M}$
for all non-zero $\bu \in \C^N$, by Corollary~\ref{Cgrquot}.

\item Our final example involves an application of Theorem
\ref{Tsimult}, in which we obtained a stratification of the cone
$\bp_N( \C )$ by the set $\Pi_N$ of partitions of
$\{ 1, \ldots, N \}$. Given a partition $\pi \in \Pi_N$, define the
stratum $\cals_\pi^{\{ 1 \}}$ as in Equation~\eqref{Estrata2}. The key
observation is that the simultaneous kernel map
$A \mapsto \mathcal{K}( A )$ is constant on each stratum. In other
words, the map
\[
\mathcal{K} : \bp_N( \C ) \longrightarrow \Pi_N \longrightarrow %
\bigsqcup_{r = 0}^{N - 1} \Gr( r, \C^N )
\]
sends every matrix $A \in \cals_\pi^{\{ 1 \}}$ to a fixed subspace
\[
\mathcal{K}_\pi := %
\ker \sum_{j = 1}^{| \pi |} {\bf 1}_{I_j \times I_j} \in %
\Gr( N - | \pi |, \C^N ),
\]
where the $N \times N$ matrix ${\bf 1}_{E \times F}$ has $( i , j )$th
entry equal to~$1$ if $( i, j) \in E \times F$ and equal to $0$
otherwise, $| \pi | = k$ denotes the number of parts in the
partition~$\pi = \{ I_1, \ldots, I_k \}$, and $\Gr( r, \C^N )$ denotes
the complex Grassmann manifold of $r$-dimensional subspaces of~$\C^N$.

Now suppose $h( z )$ is as in~\eqref{Enotation},
\[
g( z ) = \sum_{n=0}^M g_n z^n \quad ( g_n \in [ 0, \infty) ), %
\quad \text{and} \quad %
\bp = \bigsqcup_{\pi \in \Pi_N} \cals^0_\pi,
\]
where $\cals^0_\pi \subset \cals_\pi^{\{ 1 \}}$ is a compact subset of
the corresponding stratum. Then
\begin{align}\label{Estrata}
\mathcal{C}( h_\bc; g; \bp ) & \leq %
\sum_{n : g_n > 0} g_n \mathcal{C}( h_\bc; z^n; \bp ) \notag \\
 & = \sum_{n : g_n > 0} g_n \max_{\pi \in \Pi_N} %
\max\Bigl\{ \frac{\bv^* A^{\circ n} \bv}%
{\bv^* h_\bc[ A ] \bv} : \bv \in S^{2 N-  1} \cap \mathcal{K}_\pi^\perp, %
\ A \in \cals^0_\pi \Bigr\},
\end{align}
and this is a finite number because the inner maximum is taken over a
product of compact sets, and the function being optimized is
continuous in both variables.

It would be interesting to investigate the jumping locus of the map
$A \mapsto \mathcal{C}(h_\bc; z^M; A)$. In particular, is this map
continuous on the stratum $\cals_\pi^{\{ 1 \}}$, for each partition
$\pi \in \Pi_N$?

More generally, given any matrix $B \in \bp_N(\C)$ with no zero
diagonal entries, Theorem~\ref{Tsimult} provides a stratification
$\bp_N( \C ) = \bigsqcup_{\pi \in \Pi_N} \cals_{B, \pi}^{\{ 1 \}}$
in a similar vein to the above; thus
$\cals_\pi^{\{ 1 \}} = \cals_{\one{N}, \pi}^{\{ 1 \}}$
for all $\pi \in \Pi_N$. Once again, the map $\mathcal{K}$ is constant
on each stratum, sending $\cals_{B, \pi}^{\{ 1 \}}$ to
$\mathcal{K}_{B, \pi}$, say. Then, by Proposition~\ref{Ppos-constant}
and Lemma~\ref{Lalgebra}, the inequality in~\eqref{Estrata}
generalizes, for each fixed $B$, with $h_\bc[ A ]$ and $A^{\circ n}$
replaced by $B \circ h_\bc[ A ]$ and $B \circ A^{\circ n}$ for all $n$,
respectively.  See the preceding section for full details.
\end{enumerate}

\subsection*{List of symbols}

A few ad hoc notations were introduced in the text. We list them below
for the convenience of the reader.

\begin{itemize}
\item $\bp_N^k( K )$ is the set of positive semidefinite $N \times N$
matrices with entries in a subset $K \subset \C$ and of rank at most
$k$.

\item  $\bp_N( K ) := \bp_N^N( K )$.

\item $A^{\circ k}$ is the matrix obtained from $A$ by taking the
$k$th power of each entry.

\item $\one{N}$ is the $N \times N$ matrix with each entry equal
to~$1$.

\item $f[ A ]$ is the result of applying $f$ to each entry of the
matrix~$A$.

\item $\mathcal{C}( h; g; \bp )$ is the smallest non-negative
constant satisfying
$g[A] \leq \mathcal{C}( h; g; \bp ) h[ A ]$ for all $A \in \bp$.

\item $\varrho( A )$ is the spectral radius of the matrix $A$.
\end{itemize}



\begin{thebibliography}{10}

\bibitem{BCR-semigroups}
C.~Berg, J.P.R.~Christensen, and P.~Ressel.
\newblock {\em Harmonic analysis on semigroups: {T}heory of positive definite
  and related functions}, volume 100 of {\em Graduate Texts in Mathematics}.
\newblock Springer-Verlag, New York, 1984.

\bibitem{Berg-Porcu}
C.~Berg and E.~Porcu.
\newblock From Schoenberg coefficients to Schoenberg functions.
\newblock {\em Constr. Approx.}, published online, DOI:
\href{http://dx.doi.org/10.1007/s00365-016-9323-9}{10.1007/s00365-016-9323-9},
2016.

\bibitem{Bernstein}
S.~Bernstein.
\newblock Sur les fonctions absolument monotones.
\newblock \href{http://dx.doi.org/10.1007/BF02592679}{\em Acta Math.},
  52(1):1--66, 1929.

\bibitem{bickel_levina}
P.J.~Bickel and E.~Levina.
\newblock Covariance regularization by thresholding.
\newblock \href{http://dx.doi.org/10.1214/08-AOS600}{\em Ann. Statist.},
  36(6):2577--2604, 2008.

\bibitem{SIAM-optimization}
G.~Blekherman, P.A.~Parrilo, and R.R.~Thomas, editors.
\newblock {\em Semidefinite optimization and convex algebraic geometry},
  volume~13 of {\em MOS-SIAM Series on Optimization}.
\newblock Society for Industrial and Applied Mathematics (SIAM), Philadelphia,
  PA; Mathematical Optimization Society, Philadelphia, PA, 2013.

\bibitem{Bochner-pd}
S.~Bochner.
\newblock Hilbert distances and positive definite functions.
\newblock \href{http://www.jstor.org/stable/1969252}{\em Ann. of Math.
  (2)}, 42:647--656, 1941.

\bibitem{Bochner-zonal}
S.~Bochner.
\newblock Positive zonal functions on spheres.
\newblock \href{http://www.pnas.org/content/40/12.toc}{\em Proc. Nat.
  Acad. Sci. U.S.A.}, 40:1141--1147, 1954.

\bibitem{Borcea-Branden-1}
J.~Borcea and P.~Br{\"a}nd{\'e}n.
\newblock The {L}ee-{Y}ang and {P}\'olya-{S}chur programs. {I}. {L}inear
  operators preserving stability.
\newblock \href{http://dx.doi.org/10.1007/s00222-009-0189-3}{\em Invent.
  Math.}, 177(3):541--569, 2009.

\bibitem{Borcea-Branden-2}
J.~Borcea and P.~Br{\"a}nd{\'e}n.
\newblock The {L}ee-{Y}ang and {P}\'olya-{S}chur programs. {II}. {T}heory of
  stable polynomials and applications.
\newblock \href{http://dx.doi.org/10.1002/cpa.20295}{\em Comm. Pure Appl.
  Math.}, 62(12):1595--1631, 2009.

\bibitem{Branden}
P.~Br{\"a}nd{\'e}n.
\newblock The {L}ee-{Y}ang and {P}\'olya-{S}chur programs. {III}.
  {Z}ero-preservers on {B}argmann-{F}ock spaces.
\newblock \href{http://dx.doi.org/10.1353/ajm.2014.0003}{\em Amer. J.
  Math.}, 136(1):241--253, 2014.

\bibitem{Christensen_et_al78}
J.P.R.~Christensen and P.~Ressel.
\newblock Functions operating on positive definite matrices and a theorem of
  {S}choenberg.
\newblock \href{http://dx.doi.org/10.1090/S0002-9947-1978-0502895-2}{\em
  Trans. Amer. Math. Soc.}, 243:89--95, 1978.

\bibitem{Donoghue_74}
W.F.~Donoghue, Jr.
\newblock {\em Monotone matrix functions and analytic continuation}.
\newblock Springer-Verlag, New York, 1974.
\newblock Die Grundlehren der mathematischen Wissenschaften, Band 207.

\bibitem{FitzHorn}
C.H.~Fitz{G}erald and R.A.~Horn.
\newblock On fractional {H}adamard powers of positive definite matrices.
\newblock \href{http://dx.doi.org/10.1016/0022-247X(77)90167-6}{\em J.
  Math. Anal. Appl.}, 61(3):633--642, 1977.

\bibitem{Gantmacher_Vol2}
F.R.~Gantmacher.
\newblock {\em The theory of matrices. {V}ols. 1, 2}.
\newblock Translated by K.A. Hirsch. Chelsea Publishing Co., New York, 1959.

\bibitem{gneiting2013}
T.~Gneiting.
\newblock Strictly and non-strictly positive definite functions on spheres.
\newblock \href{http://dx.doi.org/10.3150/12-BEJSP06}{\em Bernoulli},
  19(4):1327--1349, 2013.

\bibitem{GMP}
J.C.~Guella, V.A.~Menegatto, and A.P.~Peron.
\newblock An extension of a theorem of {S}choenberg to products of spheres.
\newblock {\em Banach Journal of Mathematical Analysis}, 
\href{https://projecteuclid.org/accepted/euclid.bjma}{in press}
(arXiv: 1503.08174), 2016.

\bibitem{GKR-crit-2sided}
D.~Guillot, A.~Khare, and B.~Rajaratnam.
\newblock Complete characterization of {H}adamard powers preserving {L}oewner
  positivity, monotonicity, and convexity.
\newblock \href{http://dx.doi.org/10.1016/j.jmaa.2014.12.048}{\em J.
  Math. Anal. Appl.}, 425(1):489--507, 2015.

\bibitem{GKR-lowrank}
D.~Guillot, A.~Khare, and B.~Rajaratnam.
\newblock Preserving positivity for rank-constrained matrices.
\newblock {\em Trans. Amer.  Math. Soc.}, in press, DOI:
\href{http://dx.doi.org/10.1090/tran/6826}{10.1090/tran/6826}
(arXiv: 1406.0042), 2016.

\bibitem{Guillot_Khare_Rajaratnam2012}
D.~Guillot, A.~Khare, and B.~Rajaratnam.
\newblock Preserving positivity for matrices with sparsity constraints.
\newblock {\em Trans. Amer. Math. Soc.}, published online, DOI:
\href{http://dx.doi.org/10.1090/tran6669}{10.1090/tran6669}, 2016.

\bibitem{Guillot_Rajaratnam2012}
D.~Guillot and B.~Rajaratnam.
\newblock Retaining positive definiteness in thresholded matrices.
\newblock \href{http://dx.doi.org/10.1016/j.laa.2012.01.013}{\em Linear
  Algebra Appl.}, 436(11):4143--4160, 2012.

\bibitem{Guillot_Rajaratnam2012b}
D.~Guillot and B.~Rajaratnam.
\newblock Functions preserving positive definiteness for sparse matrices.
\newblock \href{http://dx.doi.org/10.1090/S0002-9947-2014-06183-7}{\em
  Trans. Amer. Math. Soc.}, 367(1):627--649, 2015.

\bibitem{HardyWright}
G.H.~Hardy and E.M.~Wright.
\newblock {\em An introduction to the theory of numbers}.
\newblock Oxford University Press, Oxford, sixth edition, 2008.
\newblock Revised by D.R. Heath-Brown and J.H. Silverman, With a foreword by
  Andrew Wiles.

\bibitem{HKKR}
H.~Helson, J.-P.~Kahane, Y.~Katznelson, and W.~Rudin.
\newblock The functions which operate on Fourier transforms.
\newblock \href{http://link.springer.com/article/10.1007/BF02559571}{\em
  Acta Math.}, 102(1):135--157, 1959.

\bibitem{HMPV}
J.W.~Helton, S.~McCullough, M.~Putinar, and V.~Vinnikov.
\newblock Convex matrix inequalities versus linear matrix inequalities.
\newblock \href{http://dx.doi.org/10.1109/TAC.2009.2017087}{\em IEEE
  Trans. Automat. Control}, 54(5):952--964, 2009.

\bibitem{hero_rajaratnam}
A.~Hero and B.~Rajaratnam.
\newblock Large-scale correlation screening.
\newblock \href{http://dx.doi.org/10.1198/jasa.2011.tm11015}{\em
  J. Amer. Statist. Assoc.}, 106:1540--1552, 2011.

\bibitem{Hero_Rajaratnam2012}
A.~Hero and B.~Rajaratnam.
\newblock Hub discovery in partial correlation graphs.
\newblock \href{http://dx.doi.org/10.1109/TIT.2012.2200825}{\em IEEE
  Trans. Inform. Theory}, 58(9):6064--6078, 2012.

\bibitem{Matrix01psd}
D.~Hershkowitz, M.~Neumann, and H.~Schneider.
\newblock Hermitian positive semidefinite matrices whose entries are {$0$} or
  {$1$} in modulus.
\newblock \href{http://dx.doi.org/10.1080/03081089908818619}{\em Linear
  and Multilinear Algebra}, 46(4):259--264, 1999.

\bibitem{Herz63}
C.S.~Herz.
\newblock Fonctions op\'erant sur les fonctions d\'efinies-positives.
\newblock
\href{http://aif.cedram.org/aif-bin/item?id=AIF_1963__13_1_161_0}{\em
  Ann. Inst. Fourier (Grenoble)}, 13:161--180, 1963.

\bibitem{Hiai2009}
F.~Hiai.
\newblock Monotonicity for entrywise functions of matrices.
\newblock \href{http://dx.doi.org/10.1016/j.laa.2009.04.001}{\em
  Lin. Alg. Appl.}, 431(8):1125--1146, 2009.

\bibitem{horn}
R.A.~Horn.
\newblock The theory of infinitely divisible matrices and kernels.
\newblock \href{http://dx.doi.org/10.1090/S0002-9947-1969-0264736-5}{\em
  Trans. Amer. Math. Soc.}, 136:269--286, 1969.

\bibitem{Kahane-Rudin}
J.-P.~Kahane and W.~Rudin.
\newblock Caract\'erisation des fonctions qui op\`erent sur les coefficients de
  {F}ourier-{S}tieltjes.
\newblock {\em C. R. Acad. Sci. Paris}, 247:773--775, 1958.

\bibitem{Li_Horvath}
A.~Li and S.~Horvath.
\newblock {Network neighborhood analysis with the multi-node topological
  overlap measure}.
\newblock
\href{http://bioinformatics.oxfordjournals.org/content/23/2.toc}{\em
  Bioinformatics}, 23(2):222--231, 2007.

\bibitem{Loewner34}
K.~L{\"o}wner.
\newblock \"{U}ber monotone {M}atrixfunktionen.
\newblock \href{http://dx.doi.org/10.1007/BF01170633}{\em Math. Z.},
  38(1):177--216, 1934.

\bibitem{Macdonald}
I.G.~Macdonald.
\newblock {\em Symmetric functions and {H}all polynomials}.
\newblock Oxford Mathematical Monographs. The Clarendon Press, Oxford
  University Press, New York, second edition, 1995.
\newblock With contributions by A. Zelevinsky, Oxford Science Publications.

\bibitem{Markus}
A.S.~Markus.
\newblock {\em Introduction to the spectral theory of polynomial operator
  pencils}, volume~71 of {\em Translations of Mathematical Monographs}.
\newblock American Mathematical Society, Providence, RI, 1988.
\newblock Translated from the Russian by H. H. McFaden, Translation edited by
  Ben Silver, With an appendix by M. V. Keldysh.

\bibitem{menegatto_2001}
V.A.~Menegatto and A.P.~Peron.
\newblock Positive definite kernels on complex spheres.
\newblock \href{http://dx.doi.org/10.1006/jmaa.2000.7264}{\em J. Math.
  Anal. Appl.}, 254(1):219--232, 2001.

\bibitem{Nemirovski}
A.~Nemirovski.
\newblock Advances in convex optimization: conic programming.
\newblock In {\em International {C}ongress of {M}athematicians. {V}ol. {I}},
  pages 413--444. Eur. Math. Soc., Z\"urich, 2007.

\bibitem{niculescu}
C.P.~Niculescu and L.-E.~Persson.
\newblock {\em Convex functions and their applications}.
\newblock CMS Books in Mathematics/Ouvrages de Math\'ematiques de la SMC, 23.
  Springer, New York, 2006.

\bibitem{Renegar}
J.~Renegar.
\newblock Hyperbolic programs, and their derivative relaxations.
\newblock \href{http://dx.doi.org/10.1007/s10208-004-0136-z}{\em Found.
  Comput. Math.}, 6(1):59--79, 2006.

\bibitem{Rothman_et_al_JASA}
A.J.~Rothman, E.~Levina, and J.~Zhu.
\newblock Generalized thresholding of large covariance matrices.
\newblock \href{http://dx.doi.org/10.1198/jasa.2009.0101}{\em
  J. Amer. Statist. Assoc.}, 104(485):177--186, 2009.

\bibitem{Rudin59}
W.~Rudin.
\newblock Positive definite sequences and absolutely monotonic functions.
\newblock \href{http://dx.doi.org/10.1215/S0012-7094-59-02659-6}{\em Duke
  Math. J}, 26(4):617--622, 1959.

\bibitem{Schoenberg42}
I.J.~Schoenberg.
\newblock Positive definite functions on spheres.
\newblock \href{http://dx.doi.org/10.1215/S0012-7094-42-00908-6}{\em Duke
  Math. J.}, 9(1):96--108, 1942.

\bibitem{Stanley}
R.P.~Stanley.
\newblock Theory and application of plane partitions. {II}.
\newblock \href{http://dx.doi.org/10.1002/sapm1971503259}{\em Stud. Appl.
  Math.}, 50(3):259--279, 1971.

\bibitem{Vinzant}
C.~Vinzant.
\newblock What is {$\ldots$} a spectrahedron?
\newblock \href{http://www.ams.org/notices/201405/}{\em Notices Amer.
  Math. Soc.}, 61(5):492--494, 2014.

\bibitem{vonNeumann-Schoenberg}
J.~von Neumann and I.J.~Schoenberg.
\newblock Fourier integrals and metric geometry.
\newblock \href{http://dx.doi.org/10.1090/S0002-9947-1941-0004644-8}{\em
  Trans. Amer. Math. Soc.}, 50:226--251, 1941.

\bibitem{Zhang_Horvath}
B.~Zhang and S.~Horvath.
\newblock A general framework for weighted gene co-expression network
analysis.
\newblock \href{http://dx.doi.org/10.2202/1544-6115.1128}{\em Stat. Appl.
  Genet. Mol. Biol.}, 4:Art. 17, 45 pp. (electronic), 2005.
\end{thebibliography}


\end{document}